\documentclass[11pt, reqno]{amsart}
\usepackage{amsmath,amsfonts,amssymb,amsthm,color,cancel}
\usepackage{epsfig}
\usepackage[normalem]{ulem}
\usepackage{ulem}

\usepackage{bbm}
\usepackage{enumerate}
\usepackage{amstext}
\usepackage{mathrsfs}
\usepackage{xspace}
\usepackage{amsfonts}
\usepackage{amsmath}
\usepackage{amssymb}
\usepackage{amstext}
\usepackage{amsthm}       
\usepackage{xspace}
\usepackage[active]{srcltx}
\usepackage{cite}

\newtheorem{theorem}{Theorem}[section]
\newtheorem{lemma}[theorem]{Lemma}

\newtheorem{definition}[theorem]{Definition}
\newtheorem{proposition}[theorem]{Proposition}
\newtheorem{remark}[theorem]{Remark}

\newcommand{\gm}{\gamma}

\newcommand{\R}{\mathbb{R}}
\newcommand{\cS}{\mathcal{S}}

\newcommand{\N}{\mathcal{N}}
\newcommand{\I}{\mathcal{I}}

\newcommand{\A}{\mathcal{A}}

\newcommand{\ri}{\rightarrow}
\newcommand{\eps}{\varepsilon}
\newcommand{\Proof}{\begin{proof}}
\newcommand{\End}{\end{proof}}
\newcommand{\EEnd}{\ensuremath{\hfill{\Box}}\\}

\numberwithin{equation}{section}

\begin{document}




\title{Aubry-Mather theory for contact Hamiltonian systems II}

\author{Kaizhi Wang \and Lin Wang \and Jun Yan}
\address{School of Mathematical Sciences, Shanghai Jiao Tong University, Shanghai 200240, China}

\email{kzwang@sjtu.edu.cn}
\address{School of Mathematics and Statistics, Beijing Institute of Technology, Beijing 100081, China}

\email{lwang@bit.edu.cn}
\address{School of Mathematical Sciences, Fudan University, Shanghai 200433, China}

\email{yanjun@fudan.edu.cn}


\begin{abstract}
\noindent In this paper, we  continue to develop Aubry-Mather and weak KAM theories for contact Hamiltonian systems $H(x,u,p)$ with certain dependence on the contact variable $u$.
 For the Lipschitz dependence case, we obtain some properties of the Ma\~{n}\'{e} set.
 For the non-decreasing case, we  provide some information on the Aubry set, such as the comparison property, graph property and
  { a partially ordered relation}  for the collection of all projected Aubry sets with respect to backward weak KAM solutions. Moreover, we find a new  flow-invariant set $\tilde{\mathcal{S}}_s$ consists of {\it strongly} static orbits, which coincides with the Aubry set $\tilde{\mathcal{A}}$ in classical Hamiltonian systems. Nevertheless, a class of examples are constructed to show
            $\tilde{\mathcal{S}}_s\subsetneqq\tilde{\mathcal{A}}$ in the contact case.
  As their applications,  we  find some new phenomena appear even if the strictly increasing dependence of $H$ on  $u$ fails  at only one point, and we  show that there is a difference for the vanishing discount problem from the negative direction between the {\it minimal} viscosity solution and {\it non-minimal} ones.
\end{abstract}

\date{\today}
\maketitle

\tableofcontents


\section{Introduction and main results}
\setcounter{equation}{0}
\subsection{Motivation and assumptions}
In the early 1990s, based on Tonelli variational principle in  Lagrangian systems, Mather \cite{M1,Mvc} founded a seminal work  on global action-minimizing orbits generated by Hamilton equations.
 The method has a profound impact on many fields, such as Hamiltonian dynamics, Hamilton-Jacobi equations and symplectic geometry (see \cite{Be,Be1,FS} for instance).
In the mid-1990's, Fathi developed weak KAM theory to study the dynamics of Hamiltonian systems and Hamilton-Jacobi equations. Weak KAM theory provides an explanation for Mather theory from  the PDE point of view and thus builds a bridge between Mather theory and viscosity solutions of  Hamilton-Jacobi equation
\begin{align}\label{hhjj}
H(x,Du)=c(H), \quad x\in M,
\end{align}
where $c(H)$ is  the critical value of $H$. Under the assumption that $H$ is a Tonelli Hamiltonian, it was shown that backward weak KAM solutions and viscosity solutions of equation (\ref{hhjj}) are the same (see \cite{CIL,CL2,LPV} for related issues on  viscosity solutions). The fundamental reference on weak KAM theory is \cite{Fat-b}.

Correspondingly, one has a natural generalization of
Hamiltonian systems, the so-called contact Hamiltonian systems \cite{Arn}. In recent years, several applications of contact Hamiltonian dynamics have been found, ranging from thermodynamics to classical and statistical mechanics. For more details on contact Hamiltonian systems, we refer the reader to  \cite{Bra,B22,BCT1,DeleonS,DeleonV,VDeleon} for a series of nice introduction and geometrical description of these systems.

 Recently, the Aubry-Mather theory for conformally symplectic systems was established by Mar\`{o} and  Sorrentino in \cite{MS}. The conformally symplectic systems are closely related to discounted Hamilton-Jacobi equations. The discounted Hamilton-Jacobi equation is a specific example of contact cases, for which the weak KAM theory was developed by Mitake and Soga in \cite{MK}. Both cases above were proceeded under Tonelli assumption and $\frac{\partial H}{\partial u}\equiv \lambda>0$.

 In \cite{WWY2},
 the authors generalized some fundamental results of Aubry-Mather theory and weak KAM theory from Hamiltonian systems to contact Hamiltonian systems under Tonelli assumption and  $0<\frac{\partial H}{\partial u}\leq \lambda$, where the analysis was  owing to a new implicit variational principle  and a dynamical approach developed in \cite{SWY,WWY,WWY1}.  It is worth mentioning that  an alternative variational formulation  was provided in \cite{CCWY,CCJWY,LTW}  in light of G. Herglotz's work  \cite{her}, which is of explicit form  with nonholonomic constraints.
See \cite{WY} for recent work on the implicit variational principle for contact Hamiltonian systems in the time-dependent and non-compact case.

In  this paper,  we continue to study the global dynamics of the contact Hamiltonian system
\begin{align}\label{c}\tag{CH}
\left\{
        \begin{array}{l}
        \dot{x}=\frac{\partial H}{\partial p}(x,u,p),\\
        \dot{p}=-\frac{\partial H}{\partial x}(x,u,p)-\frac{\partial H}{\partial u}(x,u,p)p,\qquad (x,u,p)\in T^*M\times\mathbb{R},\\
        \dot{u}=\frac{\partial H}{\partial p}(x,u,p)\cdot p-H(x,u,p).
         \end{array}
         \right.
\end{align}
under assumptions as follows. Let $H:T^*M\times\mathbb{R}\to \mathbb{R}$ be a $C^3$ function satisfying
\begin{itemize}
	\item [\textbf{(H1)}] {\it Strict convexity}:  the Hessian $\frac{\partial^2 H}{\partial p^2} (x,u,p)$ is positive definite for all $(x,u,p)\in T^*M\times\R$;
	\item [\textbf{(H2)}] {\it Superlinearity}: for every $(x,u)\in M\times\R$, $H(x,u,p)$ is  superlinear in $p$;
	\item [\textbf{(H3)}] {\it Non-decreasing}: there is a constant $\lambda>0$ such that for every $(x,u,p)\in T^{\ast}M\times\R$,
		\begin{equation*}
		0\leq \frac{\partial H}{\partial u}(x,u,p)\leq \lambda,
		\end{equation*}
\end{itemize}
where $M$ is a connected, closed and smooth Riemannian manifold, and $T^*M$ denotes the cotangent bundle of $M$. In order to deal with global dynamics of (\ref{c}), we assume additionally
\begin{itemize}
\item [\textbf{(A)}] {\it Admissibility}:  the stationary Hamilton-Jacobi equation
\begin{align}\label{hj}\tag{HJ}
H(x,u,Du)=0,\quad x\in M,
\end{align}
 has a viscosity solution.
\end{itemize}
From the dynamical point of view (\cite{WWY2}), under the assumptions (H1)-(H3), $H$ is admissible if and only if there exists $a\in \R$ such that
\[\inf_{u\in C^\infty(M,\R)}\sup_{x\in M}H(x,a,Du)=0.\]

We are devoted to extend some fundamental results of Aubry-Mather theory and weak KAM theory from Hamiltonian systems $H(x,p)$ to contact Hamiltonian systems $H(x,u,p)$ under (H1)-(H3) and (A). Different from \cite{WWY2}, this generalization is compatible with classical Hamiltonian systems (\cite{Fat-b,M1,Mvc}).
In order to deal with general cases, we have the following new challenges:
  \begin{itemize}
  \item   In classical Hamiltonian systems, due to its conservativity, the forward and backward objects can be transferred from each other. One can obtain certain ``information" by considering {\it either} forward objects {\it or} backward objects, where ``information" can be provided by some  tools like  Lax-Oleinik semigroup, weak KAM solution,  action function and  semi-static orbits  {\it etc.} In contact cases,  the correspondence between the forward and backward ``information" does not hold any more. The assumption (H3) admits both conservative and  dissipative cases. Therefore, we need to consider {\it both}  forward  {\it and} backward objects. The difficulties are closely related to the dependence of $H(x,u,p)$ on $u$. If $0<\frac{\partial H}{\partial u}\leq \lambda$, the backward Lax-Oleinik semigroup contracts along the evolution in time, while the forward Lax-Oleinik semigroup dissipates.  If $0\leq \frac{\partial H}{\partial u}\leq \lambda$, one has the {\it coexistence} of contraction, conservation and dissipation, which increases  certain difficulties to the analysis on the dynamics generated by (\ref{c}).
  \item    In view of the coupling issue of (\ref{c}),  we have to consider the action minimizing orbits and the flow invariant sets in $T^*M\times\R$ instead of $T^*M$. Moreover,  it is not natural to expect the compactness of action minimizing sets, since (backward/forward) weak KAM solutions are not unique under the assumption (H3). Due to the non-uniqueness of (backward/forward) weak KAM solutions, the action minimizing sets with respect to the solutions are also necessary to be involved.
  \end{itemize}

In this paper, we  prove the existence of  Ma\~{n}\'{e} sets $\tilde{\mathcal{N}}$,  Aubry sets $\tilde{\mathcal{A}}$ and  Mather sets $\tilde{\mathcal{M}}$. Moreover, we  provide more information on the Aubry set, such as the comparison property,  graph property and
  { a partially ordered relation}  for the collection of all projected Aubry sets with respect to backward weak KAM solutions. Compared to the definition of the static curve, we find a new  flow-invariant set $\tilde{\mathcal{S}}_s$ consists of so-called { \it strongly} static orbits, which coincides with the Aubry set $\tilde{\mathcal{A}}$ in classical Hamiltonian systems. Nevertheless, a class of examples are constructed to show
            $\tilde{\mathcal{S}}_s\subsetneqq\tilde{\mathcal{A}}$. Besides, $\tilde{\mathcal{A}}$ is not chain-recurrent in these examples.
  As applications of these results, we show
  \begin{itemize}
  \item   some new phenomena appear from both dynamical and PDE aspects, even if the strictly increasing dependence on the contact variable $u$ fails  at only one point;
  \item  there is a difference for the vanishing discount problem from the negative direction between the {\it minimal} viscosity solution and {\it non-minimal} ones.
  \end{itemize}

\subsection{Ma\~{n}\'{e} sets}
The authors \cite{WWY2} introduced the Aubry set and the Mather set for strictly increasing contact Hamiltonian systems. Due to the non-uniqueness of backward weak KAM solutions, the notion of the  Ma\~{n}\'{e} set is needed for non-decreasing cases. Following  Ma\~{n}\'{e} (\cite{CDI,CI,Mn3}),  we have

\begin{definition}[Semi-static curves]\label{semdepp}
\
\begin{itemize}
\item
A curve $(x(\cdot),u(\cdot)):[0,+\infty)\to M\times\mathbb{R}$ is called positively semi-static,  if it is positively minimizing and for each $t_1\leq t_2\in[0,+\infty)$, there holds
	\begin{equation}\label{3-3}
	u(t_2)=\inf_{s>0}h_{x(t_1),u(t_1)}(x(t_2),s).
	\end{equation}
\item
 A curve $(x(\cdot),u(\cdot)):(-\infty,0]\to M\times\mathbb{R}$ is called negatively semi-static, if it is negatively minimizing and (\ref{3-3}) holds for any $t_1$, $t_2\in(-\infty,0]$ with $t_1\leq t_2$. A curve $(x(\cdot),u(\cdot)):\mathbb{R}\to M\times\mathbb{R}$ is called  semi-static, if it is both positively and negatively semi-static.
 \end{itemize}
\end{definition}

In the case with $0<\frac{\partial H}{\partial u}\leq \lambda$, static curves are the same as semi-static ones (see Remark \ref{comtopr} below).  Let $p(t):=\frac{\partial L}{\partial \dot{x}}(x(t),u(t),\dot{x}(t))$. Then $(x(t),u(t),p(t))$ is a solution of equations (\ref{c}). We denote
\begin{itemize}
\item
 $\Phi_t:T^*M\times \R\ri T^*M\times \R$, the local flow generated by (\ref{c});
 \item  $\pi:T^*M\times\R\rightarrow M$, the standard projection once and for all.
 \end{itemize}
  Moreover, one can define
(positively, negatively) semi-static  orbit of $\Phi_t$ from semi-static curves by adding the $p(t)$-component.

\begin{definition}[Ma\~{n}\'{e} set]\label{audeine}
We call the set of all semi-static orbits the Ma\~{n}\'{e} set of $H$, denoted by $\tilde{\mathcal{N}}$.
We call $\mathcal{N}:=\pi\tilde{\mathcal{N}}$ the projected Ma\~{n}\'{e} set.
 We define $\tilde{\mathcal{N}}^+$ (resp. $\tilde{\mathcal{N}}^-$) as the set of all positively (resp. negatively) semi-static orbits.
\end{definition}

By the definition of Ma\~{n}\'{e} set, $\tilde{\mathcal{N}}$ is an invariant subset of $T^*M\times\mathbb{R}$ by $\Phi_t$. Let $\cS_-$ and $\cS_+$ be the sets of all backward weak KAM solutions and forward weak KAM solutions of $H(x,u,Du)=0$ respectively. Let $v_-\in \cS_-$, $v_+\in \cS_+$. Define
\[\tilde{\N}_{v_-}:=\tilde{\N}\cap G_{v_-},\quad \tilde{\N}_{v_+}:=\tilde{\N}\cap G_{v_+},\quad \N_{v_-}:=\pi\tilde{\N}_{v_-},\quad \N_{v_+}:=\pi\tilde{\N}_{v_+},
\]
where $G_{v_\pm}$ denote the Legendrian pseudographs of $v_\pm$, namely
\begin{equation}\label{gv--}
G_{v_\pm}:=\overline{\big\{(x,u,p): Dv_\pm(x)\ \text{exists},\  u=v_\pm(x),\ p=Dv_\pm(x)\big\}}.
\end{equation}
 By \cite[Theorem 1.1]{WWY2},
	the contact vector field generates a semi-flow $\Phi_{t}$ $(t\leq 0)$ on $G_{v_-}$ and a semi-flow $\Phi_t$ $(t\geq 0)$ on $G_{v_+}$. Define
\begin{equation}\label{sbic}
\tilde{\Sigma}_{v_-}:=\bigcap_{t\geq 0}\Phi_{-t}(G_{v_-}),\quad \tilde{\Sigma}_{v_+}:=\bigcap_{t\geq 0}\Phi_{t}(G_{v_+}),\quad \Sigma_{v_+}:=\pi\tilde{\Sigma}_{v_+},  \quad \Sigma_{v_-}:=\pi\tilde{\Sigma}_{v_-}.
\end{equation}
Given $v_-\in \cS_-$, let $ v_+(x):=\lim_{t\ri\infty}T_t^+v_-(x)$. Then $v_+\in \cS_+$. Similarly, given $v_+\in \cS_+$, let $ v_-(x):=\lim_{t\ri\infty}T_t^-v_+(x)$. Then $v_-\in \cS_-$.
Let \[\mathcal{I}_{v_-}:=\{x\in M\ |\ v_-(x)=\lim_{t\ri\infty}T_t^+v_-(x)\},\quad \mathcal{I}_{v_+}:=\{x\in M\ |\ v_+(x)=\lim_{t\ri\infty}T_t^-v_+(x)\}, \]for each $v_\pm\in \mathcal{S}_\pm$.  It follows \cite[Proposition 4.2, Remark 4.2]{WWY2} that both $v_\pm$ are of class $C^{1,1}$ on $\mathcal{I}_{v_\pm}$. Define
\begin{equation}\label{svvvec}
\tilde{\mathcal{I}}_{v_\pm}:=\{(x,u,p):\ x\in \mathcal{I}_{v_\pm},\ u=v_\pm(x),\ p=Dv_\pm(x) \}.
\end{equation}
Given $v_\pm\in \cS_\pm$, if $ v_+(x)=\lim_{t\ri\infty}T_t^+v_-(x)$ and $ v_-(x)=\lim_{t\ri\infty}T_t^-v_+(x)$, then $(v_-,v_+)$ is called a {\it conjugate pair}. $v_-$ (resp. $v_+$) is called a {\it principal} backward (resp. forward) weak KAM solution.
\begin{theorem}\label{88996}
There hold
\begin{itemize}
\item [(1)] {\bf{Covering property}}: 	\[
	\tilde{\mathcal{N}}^-=\cup_{v_-\in \mathcal{S}_-}G_{v_-},\quad \tilde{\mathcal{N}}^+=\cup_{v_+\in \mathcal{S}_+}G_{v_+}.
	\]
Moreover, $\pi \tilde{\N}^{\pm}=M$.
\item [(2)] {\bf Local characterizations}:
\[\tilde{\N}_{v_-}=\tilde{\mathcal{I}}_{v_-}=\tilde{\Sigma}_{v_-},\quad \tilde{\N}_{v_+}=\tilde{\mathcal{I}}_{v_+}=\tilde{\Sigma}_{v_+},\]
where $\tilde{\mathcal{I}}_{v_\pm}$ and $\tilde{\Sigma}_{v_\pm}$ are defined as (\ref{sbic}) and (\ref{svvvec}) below. Moreover, if $(v_-,v_+)$ is a  conjugate pair, then
\[\tilde{\N}_{v_-}=\tilde{\N}_{v_+}=G_{v_-}\cap G_{v_+}.\]
\item [(3)]  {\bf Global characterizations}: $\tilde{\N}^{\pm}$ and $\tilde{\N}$ are non-empty, closed and
\[\tilde{\N}=\cup_{v_-\in \cS_-}\tilde{\N}_{v_-}=\cup_{v_+\in \cS_+}\tilde{\N}_{v_+}.\]
\end{itemize}
\end{theorem}

\begin{remark}\label{lipcondi}
Theorem \ref{88996} still holds if the assumption {\rm (H3)} is replaced by $|\frac{\partial H}{\partial u}|\leq \lambda$. We  prove this theorem under the assumption $|\frac{\partial H}{\partial u}|\leq \lambda$ instead of {\rm (H3)}.
\end{remark}

To prove Theorem \ref{88996}, a key ingredient is to construct a (backward/forward) weak KAM solution by a (negatively/positively) semi-static curve. Since the non-expensiveness of (backward/forward) Lax-Oleinik semigroup $T^{\pm}_t$ fails under the assumption $|\frac{\partial H}{\partial u}|\leq \lambda$, it brings certain difficulties for the analysis on the asymptotic behavior of $T^{\pm}_t$. To deal with this issue, we need to pay more attention on the evolution of the action function along  semi-static curves.
\subsection{Aubry sets and Mather sets}

By definition, for each $v_-\in \cS_-$, $\tilde{\mathcal{N}}_{v_-}$ is an invariant subset of $T^*M\times\mathbb{R}$ by $\Phi_t$.
By Theorem \ref{88996},  $\tilde{\mathcal{N}}_{v_-}$ is non-empty and compact. Consequently,
there exist Borel $\Phi_t$-invariant probability measures supported in $\tilde{\mathcal{N}}_{v_-}$, called {\it Mather measures}. Denote by $\mathfrak{M}$ the set of Mather measures for all $v_-\in \cS_-$. {\it Mather set} of contact Hamiltonian systems (\ref{c}) is defined by
\[\tilde{\mathcal{M}}=\mathrm{cl}\left(\bigcup_{\mu\in \mathfrak{M}}\text{supp}(\mu)\right).\]

By the definitions of Aubry set and Mather set, both of them are invariant subsets of $T^*M\times\mathbb{R}$ by $\Phi_t$. Define
\[\tilde{\mathcal{A}}_{v_-}:=\tilde{\A}\cap G_{v_-},\quad \tilde{\A}_{v_+}:=\tilde{\A}\cap G_{v_+},\quad\A_{v_-}:=\pi\tilde{\A}_{v_-},  \quad \A_{v_+}:=\pi\tilde{\A}_{v_+},
\]
\[\tilde{\mathcal{M}}_{v_-}:=\tilde{\mathcal{M}}\cap G_{v_-},\quad \tilde{\mathcal{M}}_{v_+}:=\tilde{\mathcal{M}}\cap G_{v_+},\quad\mathcal{M}_{v_-}:=\pi\tilde{\mathcal{M}}_{v_-},  \quad \mathcal{M}_{v_+}:=\pi\tilde{\mathcal{M}}_{v_+}.
\]

By definitions, $\tilde{\mathcal{M}}\subseteq \tilde{\N}^{\pm}$ and $\tilde{\A}\subseteq \tilde{\N}^{\pm}$. Moreover,  we obtain the following.

\begin{theorem}\label{namss}
\
\begin{itemize}
\item [(1)] {\bf Characterizations}:
\begin{align*}
\tilde{\A}=\cup_{v_-\in \cS_-}\tilde{\A}_{v_-}=\cup_{v_+\in \cS_+}\tilde{\A}_{v_+},\quad\tilde{\mathcal{M}}=\cup_{v_-\in \cS_-}\tilde{\mathcal{M}}_{v_-}=\cup_{v_+\in \cS_+}\tilde{\mathcal{M}}_{v_+}.
\end{align*}
\item [(2)] {\bf Asymptotic behavior}:
For each $v_-\in \cS_-$ and each \[(x_1,v_1,p_1)\in \tilde{\mathcal{N}}_{v_-}^-,\quad (x_2,v_2,p_2)\in \tilde{\mathcal{N}}_{v_-}^+,\] the $\alpha$-limit set of $(x_1,v_1,p_1)$ and $\omega$-limit set of $(x_2,v_2,p_2)$ are contained in $\tilde{\mathcal{A}}_{v_-}$. Moreover, for each $(x_1,v_1,p_1)\in \tilde{\mathcal{N}}^-$, $(x_2,v_2,p_2)\in \tilde{\mathcal{N}}^+$, the $\alpha$-limit set of $(x_1,v_1,p_1)$ and $\omega$-limit set of $(x_2,v_2,p_2)$ are contained in $\tilde{\mathcal{A}}$.
\item [(3)] {\bf Inclusion relation}: Aubry set $\tilde{\A}$ and Mather set $\tilde{\mathcal{M}}$ are closed. Moreover,
\[\emptyset\neq\tilde{\mathcal{M}}\subseteq \tilde{\mathcal{A}}\subseteq\tilde{\mathcal{N}}.\]
\end{itemize}
\end{theorem}


\subsection{More on Aubry sets}
In this part, we  provide more information on Aubry sets. The first one is concerned with the  comparison property of $u_-\in \cS_-$.

\begin{theorem}[\bf Comparison property]\label{crogrpccc}
 Given  $u_-,v_-\in \cS_-$, if $u_-\leq v_-$ on $\A_{v_-}$, then $u_-\leq v_-$ on $M$. Moreover, if $u_-=v_-$ on $\A_{u_-}\cup\A_{v_-}$, then $u_-=v_-$ on $M$.
\end{theorem}

In fact, one has a corresponding result in terms of the Mather set (see Remark \ref{macom} below).   By a completely different approach, Jing, Mitake and Tran \cite{JMT} introduced adjoint measures and they proved that the closure of the union of all supports of those measures  plays a role as the uniqueness set of the HJ equations. It is also remarked that those measures turn out to be projected Mather measures if $H$ satisfies appropriate conditions.

Let $\bar{\mathcal{A}}$ be the dual Aubry set of $\tilde{\A}$ in $TM\times\R$ by the Legendre transformation $p=\frac{\partial L}{\partial \dot{x}}(x,u,\dot{x})$. Let \[\rho: TM\times\R \ri TM,\quad \rho^*: T^*M\times\R \ri T^*M,\quad\Pi: TM\ri M,\quad\Pi^*: T^*M\ri M\] be the standard projections. The second one is about the graph property of $\Pi$ restricting on $\rho(\bar{\A})$.
\begin{theorem}[\bf Graph property]\label{grappff}
	 $\Pi: \rho(\bar{\A})\ri M$ is injective.
\end{theorem}
 In order to prove it, we need to show that for  each static orbit $(x(\cdot),u(\cdot),\dot{x}(\cdot)):\R\ri TM\times\R$, $\dot{x}(t)$ is uniquely determined by $x(t)$ for each $t\in \R$. In view of the complicated structure of $\cS_-$ under the assumption (H3), the injectivity of $\Pi$ is subtle to be confirmed. In addition, due to the appearance of $u$-component in the Legendre transformation $p=\frac{\partial L}{\partial \dot{x}}(x,u,\dot{x})$, one can not expect the injectivity of  the standard projection $\Pi^*$ from  $\rho^*(\tilde{\A})$ to $M$, except for the separated case with $H(x,u,p):=f(x,u)+h(x,p)$.

The third one is about the structure of the projected Aubry set $\A$. Note that $\A=\cup_{u_-\in \cS_-}\A_{u_-}$. We obtain a partially ordered relation for the collection $\{\A_{u_-}\}_{u_-\in \cS_-}$.

\begin{theorem}[\bf Partially ordered relation]\label{grpaord}
	 Given  $u_-,v_-\in \cS_-$, if $u_-\leq  v_-$ on $\A_{u_-}$, then ${\mathcal{A}}_{u_-}\subseteq \mathcal{A}_{v_-}$.
\end{theorem}

Comparably,
\begin{itemize}
\item [(1)]
	for classical Hamiltonians $H(x,p)$, the Aubry set $\tilde{\mathcal{A}}$ is independent of viscosity solutions, in this case, there always hold ${\mathcal{A}}={\mathcal{A}}_{u_-}={\mathcal{A}}_{v_-}$ for any $u_-,v_-\in \mathcal{S}_-$;
\item [(2)]
for the case with  discounted Hamiltonians $\lambda u+H(x,p)$, $\lambda>0$ or more generally, strictly increasing Hamiltonians $0<\frac{\partial H}{\partial u}\leq \lambda$, we have ${\mathcal{A}}={\mathcal{A}}_{u_-}$;
\end{itemize}
In general cases,  a new phenomenon with ${\mathcal{A}}_{u_-}\subsetneqq {\mathcal{A}}_{v_-}$ for certain $u_-,v_-\in \mathcal{S}_-$ could emerge (see Proposition \ref{exammmmpp} below).

\subsection{Strongly static set}

Compared to the definition of the static curve, it is natural to consider another kind of action minimizing curves.
\begin{definition}[Strongly static curves]\label{stcontx}
	A curve $(x(\cdot),u(\cdot)):\mathbb{R}\to M\times\mathbb{R}$ is called strongly static, if it is globally minimizing and for each $t_1, t_2\in\mathbb{R}$, there holds
	\begin{equation}\label{3-399x}
	u(t_2)=\sup_{s>0}h^{x(t_1),u(t_1)}(x(t_2),s).
\end{equation}
\end{definition}
In the same way as Definition \ref{audeine}, one can define {\it strongly static orbits} $(x(\cdot),p(\cdot),u(\cdot)):\R\ri T^*M\times\R$. Moreover, one has a new flow-invariant set defined as follow.
\begin{definition}[Strongly static set]\label{audeine9911}
We call the closure of the set of all strongly static orbits the strongly static set of $H$, denoted by $\tilde{\mathcal{S}}_s$. We call $\mathcal{S}_s:=\pi\tilde{\mathcal{S}}_s$ the projected strongly static set.
\end{definition}

\begin{theorem}[\bf Strong staticity and staticity]\label{conterstatic}There holds
\[\tilde{\mathcal{S}}_s\subseteq \tilde{\mathcal{A}}.\]
In particular, $\tilde{\mathcal{S}}_s=\tilde{\mathcal{A}}$ if $H$ is independent of $u$.
 \end{theorem}

 Generally, $\tilde{\mathcal{S}}_s$  may be  a proper subset of $ \tilde{\mathcal{A}}$. For example, we have
\begin{proposition}[\bf Difference between $\tilde{\mathcal{S}}_s$ and $\tilde{\mathcal{A}}$]\label{exppp}
Let
\begin{equation}\label{exhxxx}\tag{E1}
H(x,u,p):=\lambda u+\frac{1}{2}|p|^2+p\cdot V(x),\quad (x,u,p)\in T^*\mathbb{T}\times\R,
\end{equation}
where $\mathbb{T}$ denotes a flat circle and $V:\mathbb{T}\ri \R$ is a $C^3$ function which has exactly two vanishing points $x_1$, $x_2$ with $V'(x_1)>0$, $V'(x_2)<0$.
Then
\[\tilde{\A}=\left\{(x,0,0)\ |\ x\in \mathbb{T}\right\},\quad \tilde{\mathcal{S}}_s=\left\{(x_1,0,0),\ (x_2,0,0)\right\}.\]
\end{proposition}

From the conditions on $V$, it is clear that $\tilde{\mathcal{A}}$ is not chain-recurrent. Nevertheless, $\tilde{\mathcal{S}}_s$ is non-wandering.

\subsection{Applications}
First of all, we give an example to show  certain new phenomena appearing in  contact Hamiltonian systems with non-decreasing dependence on $u$.
\begin{equation}\label{exammm}\tag{E2}
H(x,u,p):=\frac{1}{2}|p|^2+f(x)u, \ \quad (x,u,p)\in T^*\mathbb{T}\times\mathbb{R},
\end{equation}
where $\mathbb{T}:=\mathbb{R}/\mathbb{Z}$ denotes a flat circle, $f:\mathbb{T}\to\mathbb{R}$ is a smooth function with $f(0)=0$ and $f(x)>0$ for all $x\in [-\frac{1}{2},\frac{1}{2})\backslash\{0\}$, where $[-\frac{1}{2},\frac{1}{2})$ is a fundamental domain of $\mathbb{T}$.
It is clear that $u_-(x)\equiv 0$ is a solution of $H(x,u,Du)=0$.
\begin{proposition}\label{exammmmpp}
For the Hamilton-Jacobi equation
\[\frac{1}{2}|Du|^2+f(x)u=0,\quad x\in \mathbb{T},\]
there exist an uncountable family of nontrivial viscosity solutions  $\{v_i\}_{i\in I}$, and
$\mathcal{A}_{v_i}\subsetneqq \mathcal{A}_{u_-}$ for each $i\in I$.
\end{proposition}

The vanishing discount problem for discounted Hamilton-Jacobi equations has been widely studied by using dynamical and PDE approaches, see  \cite{Dav,Go,Go1,It,i1,i2,MT,S}. Recently, some new progresses have been made on the vanishing contact structure problem (see \cite{C,CCIZ,WYZ,ZC}) and the vanishing discount problem from the negative direction \cite{dw}.
Consider the HJ equation
\begin{equation}\label{vfrnn}
\lambda u_\lambda+H(x,Du_\lambda)=c(H), \ \  \ \ x\in M,
\end{equation}
it was shown in \cite{dw} that if $H$ is a $C^3$ Tonelli Hamiltonian satisfying $H(x,0)\leq c(H)$, then the minimal viscosity solution converges uniformly, as $\lambda\ri 0^-$, to the same viscosity solution of $H(x,Du)=c(H)$ like the convergence as  $\lambda\ri 0^+$ established in \cite{Dav}. It is natural to ask {\it what happens for non-minimal viscosity solutions}.
As an application of the strongly static set, we give a partial answer to this question. In light of (\ref{exhxxx}), we have
\begin{proposition}\label{exvnish}
 Let $V:\mathbb{T}\ri \R$  be a $C^3$ function which has exactly two vanishing points $x_1$, $x_2$ with $V'(x_1)>0$, $V'(x_2)<0$. Then $u^+_\lambda\equiv 0$ is the maximal forward weak KAM solution of
\begin{equation}\label{exhjj}\tag{HJE}
\lambda u+\frac{1}{2}|Du|^2+Du\cdot V(x)=0,\quad x\in \mathbb{T},
\end{equation}
 which converges uniformly to $u\equiv 0$ as $\lambda\ri 0^+$. Moreover, there is only one forward weak KAM solution $v^+_\lambda$, but  $v^+_\lambda$  converges uniformly to $v$  with $v(x_2)<0$ as $\lambda\ri 0^+$.
\end{proposition}

\begin{remark}
If we take $V(x):=\sin x$ in {\rm(\ref{exhjj})}, then $x_1=0$, $x_2=\pi$. A direct calculation shows  $v^+_\lambda$  converges uniformly to $v(x):=2\cos x-2$ as $\lambda\to 0^+$.
\end{remark}

Let us recall the correspondence between the viscosity solution and forward weak KAM solution \cite[Proposition 2.8]{WWY2}:\begin{quote}
$u_\lambda$ is a viscosity solution of equation \eqref{vfrnn} if and only if $-u_\lambda$ is a forward weak KAM solution of equation $-\lambda u_\lambda+H(x,-Du_\lambda)=c(H)$.
\end{quote}
By Proposition \ref{exvnish} and the correspondence,  we know that there is a difference between the asymptotic behavior of the {\it minimal} viscosity solution and {\it non-minimal} ones for the vanishing discount problem from the negative direction.

The paper is organized as follows. In Section \ref{mset}, we prove Theorem \ref{88996}, for which we need a technical lemma (Lemma \ref{ke}). Its proof is given in Appendix \ref{prke} and  some useful facts on the semi-static curve are collected in Appendix \ref{ossc}. In Section \ref{asetmset},  we prove Theorem \ref{namss}. In Section \ref{maset}, we prove Theorem \ref{crogrpccc} and Theorem \ref{grpaord}. For the consistency, the proof of Lemma \ref{kkkee} is postponed to Appendix \ref{prkkkee}. In Section \ref{ssaset}, we prove  Theorem \ref{conterstatic} and Proposition \ref{exppp}. In Section \ref{appl}, we prove Proposition \ref{exammmmpp} and Proposition \ref{exvnish}.
 For the sake of completeness, we recall some basic tools for contact Hamiltonian systems in Appendix \ref{gene}.

\section{Ma\~{n}\'{e} sets}\label{mset}

In this part, we  prove Theorem \ref{88996}. First of all, we need to construct a backward  solution of $H(x,u,Du)=0$ by using negatively semi-static curves. As complement, we provide some facts on the semi-static curve in Appendix \ref{ossc}. These facts can be obtained by certain arguments similar to \cite{WWY2}.

\subsection{A Busemann backward weak KAM solution}

\begin{lemma}\label{ke}
Let $(x(\cdot),u(\cdot)):(-\infty,0]\to M\times\mathbb{R}$ be a negatively semi-static curve. Let
\begin{equation}\label{buss}
w(x):=\inf_{\tau\geq 0}\inf_{s>0}h_{x(-\tau),u(-\tau)}(x,s),\quad \forall x\in M.
\end{equation}
There hold
\begin{itemize}
\item [(i)] $w$ is Lipschitz continuous on $M$;
\item [(ii)] the uniform limit $\lim_{t\rightarrow +\infty}T_t^-w(x)$ exists, let $w_\infty(x):=\lim_{t\rightarrow +\infty}T_t^-w(x)$, then $w_\infty\in \cS_-$;
\item [(iii)] $w_\infty(x(-\tau))=u(-\tau)$ for each $\tau\geq 0$.
\end{itemize}
\end{lemma}

The proof of Lemma \ref{ke} is  standard.
For the sake of consistency, we  give it in Appendix \ref{prke}.
For a classical Hamiltonian $H(x,p)$, $w(x)$ defined in (\ref{buss}) can be reduced to a Busemann backward weak KAM solution up to a constant $u(0)$. In fact, by \cite[Proposition 4-9.7]{CI}, a Busemann backward weak KAM solution is given by
\[u_b(x):=\inf_{\tau\geq 0}\{\Phi(x(-\tau),x)-\Phi(x(-\tau),x(0))\},\]
where $x(\cdot):(-\infty,0]\ri M$ is a negatively semi-static curve and $\Phi(y,x)$ denotes the Ma\~{n}\'{e} potential. In the case with a classical Hamiltonian $H(x,p)$,  for each constant $c\in \R$, we have the relation
\[\Phi(y,x)=\inf_{s>0}h_{y,c}(x,s)-c.\]
Compared to (\ref{buss}),  there holds
\[w(x)=u_b(x)+u(0).\]

\subsection{Covering property of $\tilde{\mathcal{N}}^{\pm}$}

In this part, we  prove Theorem \ref{88996}(1):
	\[
	\tilde{\mathcal{N}}^-=\cup_{v_-\in \mathcal{S}_-}G_{v_-},\quad \tilde{\mathcal{N}}^+=\cup_{v_+\in \mathcal{S}_+}G_{v_+}.
	\]

\begin{proposition}\label{4equ}
	For each $v_-\in \mathcal{S}_-$, we have $G_{v_-}\subseteq \tilde{\mathcal{N}}^-$. In particular, $\pi\tilde{\mathcal{N}}^-=M$.
\end{proposition}

\begin{proof}
	For each $(x_0,u_0,p_0)\in G_{v_-}$ and $t\geq 0$,
	let \[(x(-t),u(-t),p(-t)):=\Phi_{-t}(x_0,u_0,p_0).\] Since $G_{v_-}$ is invariant by $\Phi_{-t}$ for each $t\geq 0$, we have $u(-t)=v_-(x(-t))$ for all $t\geq 0$.	By (\ref{c}), we have
\[u(0)-u(-t)=\int_{-t}^0L(x(\tau),u(\tau),\dot{x}(\tau))d\tau,\]
which implies for each $t>0$,
\[v_-(\gm(0))-v_-(x(-t))=\int_{-t}^0L(\gm(\tau),v_-(\gm(\tau)),\dot{\gm}(\tau))d\tau.\]
It follows that $\gm:(-\infty,0]\ri M$ is ($v_-,L,0$)-calibrated. By Proposition \ref{pr1188}, $(x(\cdot),u(\cdot)):(-\infty,0]\ri M\times\mathbb{R}$ is a negatively semi-static curve.
	Let $p(t)=\frac{\partial L}{\partial \dot{x}}(x(t),u(t),\dot{x}(t))$.
	Thus, \[(x(-t),u(-t),p(-t))\in \tilde{\mathcal{N}}^-\] for each $t\geq 0$. In particular, $(x_0,u_0,p_0)\in \tilde{\mathcal{N}}^-$. Note that $v_-$ is Lipschitz continuous, we have $\pi G_{v_-}=M$. Thus, $\pi \tilde{\mathcal{N}}^-=M$.
\end{proof}

\begin{lemma}\label{dstar}
	Let $v_-\in \mathcal{S}_-$. Then
	\[
	D^*v_-(x)=\{p\in D^+v_-(x)\ |\ H(x,v_-(x),p)=0\},
	\]
 where we use $D^*u(x)$ to denote the set of all reachable gradients of $u$ at $x$.
\end{lemma}
The proof of the lemma is similar to the one of Theorem 6.4.12 in \cite{CS}, we omit it here for brevity. The following lemma is similar to \cite[Lemma 4.3]{WWY2}.

\begin{lemma}\label{diffe11}
	Let $u\in C(M,\mathbb{R})$, $x\in M$ and $t>0$. If $\gamma:[0,t]\rightarrow M$ is a minimizer of $T^-_tu(x)$, i.e.,
	\[
	T^-_tu(x)=u(\gamma(0))+\int_0^tL(\gamma(\tau),T^-_\tau u(\gamma(\tau)),\dot{\gamma}(\tau))d\tau,
	\]
	then
	\[
	\frac{\partial L}{\partial \dot{x}}(\gamma(t),T^-_tu(\gamma(t)),\dot{\gamma}(t))\in D^+[T_t^-u](\gamma(t)).	
	\]
In particular, if $\gm:(-\infty,0]\ri M$ is ($v_-,L,0$)-calibrated, then for all $t\geq 0$,
\[
	\frac{\partial L}{\partial \dot{x}}(\gamma(-t),v_-(\gamma(-t)),\dot{\gamma}(-t))\in D^+v_-(\gamma(-t)).	
	\]
\end{lemma}

\noindent{\bf Proof of Theorem \ref{88996}(1).}
By Proposition \ref{4equ}, we have $\cup_{v_-\in \mathcal{S}_-}G_{v_-}\subseteq\tilde{\mathcal{N}}^-$.
Thus, it suffices to prove $\cup_{v_-\in \mathcal{S}_-}G_{v_-}\supseteq\tilde{\mathcal{N}}^-$.
For each $(x_0,u_0,p_0)\in \tilde{\mathcal{N}}^-$, we need to find a $v_-\in \mathcal{S}_-$ such that  $p_0\in D^*v_-(x_0)$ and $u_0=v_-(x_0)$.

Let $(x(-\tau),u(-\tau),p(-\tau)):=\Phi_{-\tau}(x_0,u_0,p_0)$ with $(x_0,u_0,p_0)=(x(0),u(0),p(0))$ for each $\tau\geq 0$. Let $\gm(t):=x(t)$ for each $t\in (-\infty,0]$, By Proposition \ref{pr1188}, there exists $v_-\in \cS_-$ such that $\gm:(-\infty,0]\ri M$ is ($v_-,L,0$)-calibrated.
In view of \cite[Proposition 4.1]{WWY2}, we have for all $\tau\leq 0$, \[H(\gm(-\tau),u(-\tau),p(-\tau))=0.\]  In particular, $H(x_0,u_0,p_0)=0$. From Lemma \ref{dstar} and Lemma \ref{diffe11}, we have
$p_0\in D^*v_-(x_0)$.
This completes the proof.
\hfill$\Box$

\subsection{Local characterizations}

In \cite[Lemma 4.7,Remark 4.2]{WWY2}, we obtain that
\begin{lemma}\label{iinv}
	For each $v_-\in \cS_-$, let $v_+:=\lim_{t\ri +\infty}T_t^+v_-$.
	For any given $x\in M$ with $v_-(x)=v_+(x)$,  there exists a curve $\gamma:\R\rightarrow M$ with $\gamma(0)=x$ such that  $v_-(\gamma(t))=v_+(\gamma(t))$ for each $t\in \mathbb{R}$, and
	\begin{equation}\label{upmm}
	v_{\pm}(\gamma(t'))-v_{\pm}(\gamma(t))=\int_t^{t'}L(\gamma(s),v_{\pm}(\gamma(s)),\dot{\gamma}(s))ds, \quad \forall t\leq t'\in\mathbb{R}.
	\end{equation}
	Moreover, $v_{\pm}$ are differentiable at $x$ with the same derivative $Dv_{\pm}(x)=\frac{\partial L}{\partial \dot{x}}(x,v_{\pm}(x),\dot{\gamma}(0))$.
\end{lemma}

\medskip
\noindent{\bf Proof of Theorem \ref{88996}(2).}  We  prove
for each $v_-\in\mathcal{S}_-$, $v_+\in\mathcal{S}_+$,
	\[
	\tilde{\mathcal{N}}_{v_-}=\tilde{\Sigma}_{v_-}=\tilde{\I}_{v_-},\quad \tilde{\mathcal{N}}_{v_+}=\tilde{\Sigma}_{v_+}=\tilde{\I}_{v_+}.
	\]

We only need to show the first part, since the second one can be obtained by a similar argument.
	For any given $v_-\in\mathcal{S}_-$, let
	$v_+:=\lim_{t\rightarrow +\infty}T^+_tv_-$.
	We first show that $\tilde{\mathcal{N}}_{v_-}\subseteq\tilde{\Sigma}_{v_-}$. For each $(x_0,u_0,p_0)\in \tilde{\mathcal{N}}_{v_-}$, since $\tilde{\mathcal{N}}_{v_-}$ is invariant by $\Phi_t$, we have
	\[(x(t),u(t),p(t))=\Phi_{t}(x_0,u_0,p_0)\in \tilde{\mathcal{N}}_{v_-}\subseteq G_{v_-}\] for all $t\in \mathbb{R}$. Thus, we have $(x_0,u_0,p_0)\in \Phi_{-t}(G_{v_-})$ for all $t\in \mathbb{R}$, which implies
	\[
	(x_0,u_0,p_0)\in \cap_{t\leq 0}\Phi_t(G_{v_-})=\tilde{\Sigma}_{v_-}.
	\]
	Since $v_-(x)=v_+(x)$ for all $x\in\Sigma_{v_-}$ (\cite[Proposition 4.5]{WWY2}), then by Lemma \ref{iinv} we get
	\begin{align}\label{5-203}
	\tilde{\Sigma}_{v_-}\subseteq \tilde{\I}_{v_-}.
	\end{align}
Next, we show $\tilde{\I}_{v_-}\subseteq\tilde{\mathcal{N}}_{v_-}$.
	For each $(x,u,p)\in \tilde{\I}_{v_-}$, there exists $v_+:=\lim_{t\ri +\infty}T_t^+v_-$
	such that $v_-(x)=v_+(x)=u$, $p=Du_{\pm}(x)$.
	By Lemma \ref{iinv}, there exists a curve $\gamma:\R\rightarrow M$ with $\gamma(0)=x$ such that $v_-(\gamma(t))=v_+(\gamma(t))$ for all $t\in\mathbb{R}$ and  $p=\frac{\partial L}{\partial \dot{x}}(x,u,\dot{\gamma}(0))$. Let $x(t):=\gm(t)$, $u(t):=v_-(\gamma(t))$. By Proposition \ref{pr119955},  $(x(\cdot),u(\cdot)):\R\ri M\times\mathbb{R}$ is a semi-static curve.

By definitions, if $(v_-,v_+)$ is a  conjugate pair, then \[\tilde{\mathcal{I}}_{v_-}=\tilde{\mathcal{I}}_{v_+}=G_{v_-}\cap G_{v_+}.\]
The completes the proof of Theorem \ref{88996}(2).\EEnd

\subsection{Global characterizations}
By Theorem \ref{88996}(1) and $\tilde{\mathcal{N}}\subseteq \tilde{\N}^{\pm}$, we have
\[\tilde{\N}=\cup_{v_-\in \cS_-}\tilde{\N}_{v_-}=\cup_{v_+\in \cS_+}\tilde{\N}_{v_+}.\]

	It remains to prove $\tilde{\N}$ is closed, since the closedness of $\tilde{\N}^{\pm}$ are similar to be obtained. Let $\{(x_n,u_n,p_n)\}_{n\in \mathbb{N}}\subseteq \tilde{\mathcal{N}}$ with
	$(x_n,u_n,p_n)\to (x_0,u_0,p_0)$ as $n\to\infty$. It is sufficient to show that $(x_0,u_0,p_0)\in \tilde{\mathcal{N}}$. Let
\[(x_0(t),u_0(t),p_0(t)):=\Phi_t(x_0,u_0,p_0).\]
For each $n\in \mathbb{N}$, let \[(x_n(t),u_n(t),p_n(t)):=\Phi_t(x_n,u_n,p_n).\]
Note that $(x_n,u_n,p_n)\ri (x_0,u_0,p_0)$ as $n\ri +\infty$. By Proposition \ref{ooo}, $\Phi_t(x_n,u_n,p_n)$ is uniformly bounded for each $n\in \mathbb{N}$.
By the continuity and differentiability of solutions of ordinary differential equations with respect to initial values,	the solution $(x_0(t),u_0(t),p_0(t))$  exists for all $t\in \R$. Thus, we only need to prove $(x_0(t),u_0(t))$ is a semi-static curve.

 For each $t_1$, $t_2\in\R$ with $t_1\leq t_2$, take $K>0$ such that $[t_1,t_2]\in[-K,K]$.
	Note that the sequence of $(x_n(t),u_n(t),p_n(t))$ converges uniformly to  $(x_0(t),u_0(t),p_0(t))$ as $n\to \infty$ on $[-K,K]$. Thus, we get $(x_n(t_i),u_n(t_i),p_n(t_i))\to (x_0(t_i),u_0(t_i),p_0(t_i))$ as $n\to \infty$ for $i=1$, $2$.

	Fix $s>0$ and by the Lipschitz continuity of $h_{x_0,u_0}(x,t)$ w.r.t. $x_0,u_0,x$, for each $\varepsilon>0$ there is $N\in\mathbb{N}$ such that
	\[
	|h_{x_n(t_1),u_n(t_1)}(x_n(t_2),s)-h_{x_0(t_1),u_0(t_1)}(x_0(t_2),s)|\leq \varepsilon,\quad \forall n>N.
	\]
	Since $(x_n(t),u_n(t))$ is semi-static, then we have
	\[
	h_{x_0(t_1),u_0(t_1)}(x_0(t_2),s)\geq h_{x_n(t_1),u_n(t_1)}(x_n(t_2),s)-\epsilon\geq  u_n(t_2)-\epsilon.
	\]
	Letting $\varepsilon\to 0$ and $n\to\infty$, we get
	\[
	h_{x_0(t_1),u_0(t_1)}(x_0(t_2),s)\geq u_0(t_2).
	\]
	Hence, $\inf_{s>0}h_{x_0(t_1),u_0(t_1)}(x_0(t_2),s)\geq u_0(t_2)$.
	On the other hand, by the minimality property of $h_{x_0,u_0}(x,t)$, we have
	\[
	h_{x_0(t_1),u_0(t_1)}(x_0(t_2),t_2-t_1)\leq u_0(t_1),
	\]
	which gives rise to $\inf_{s>0}h_{x_0(t_1),u_0(t_1)}(x_0(t_2),s)\leq u_0(t_2)$.
	
Since $(x_n(t),u_n(t))$ is globally minimizing, then it is straightforward to see that $(x_0(t),u_0(t))$ is also a globally minimizing curve. Therefore, $(x_0(t),u_0(t))$ is semi-static.

This completes the proof of of Theorem \ref{88996}(3).

\section{Aubry sets and Mather sets}\label{asetmset}
This part is devoted to proving Theorem \ref{namss}. Item (1) follows from Theorem  \ref{88996}(1) directly. We prove Item (2) and Item (3) in the following. At the beginning, we  generalize Peierls barrier and Ma\~{n}\'{e} potential to contact Hamiltonian systems and show their uniformly Lipschitz properties.
\begin{proposition}[\cite{SWY}]\label{con}
	For each $\varphi\in C(M,\mathbb{R})$, the uniform limit $\lim_{t\rightarrow +\infty}T^-_t\varphi(x)$ exists. Let \[u_\infty(x):=\lim_{t\rightarrow +\infty}T^-_t\varphi(x)\] for each  $x\in M$. Then $u_\infty(x)$ is a viscosity solution of equation (\ref{hj}).
\end{proposition}
 Thus, the following function
\begin{align}\label{ba}
h_{x_0,u_0}(x,+\infty):=\lim_{t\rightarrow +\infty}h_{x_0,u_0}(x,t),\quad x\in M
\end{align}
is well defined. It can be viewed as {\it Peierls barrier} of (\ref{c}). Accordingly, $\inf_{s>0}h_{x_0,u_0}(x,s)$ is also well defined, which can be viewed as  {\it Ma\~{n}\'{e} potential} of (\ref{c}).

\begin{proposition}\label{aclipp}
	The function $(x_0,u_0,x)\mapsto h_{x_0,u_0}(x,+\infty)$ is uniformly Lipschitz  continuous on $M\times\R\times M$.
\end{proposition}
\begin{proof}
Given $(x_0,x)\in M\times M$, by Proposition \ref{nonehh}, $u_0\mapsto h_{x_0,u_0}(x,+\infty)$ is uniformly Lipschitz  continuous on $\R$.

Next, we  prove that $x\mapsto h_{x_0,u_0}(x,+\infty)$ is Lipschitz continuous. By \cite[Theorem B.1]{WWY2}, for a given $\delta>0$, $(x,t)\mapsto h_{x_0,u_0}(x,t)$ is bounded by $K>0$ on $M\times [\delta,+\infty)$. Note that for any $t> 2\delta$, we have
\begin{align*}
&\left|h_{x_0,u_0}(x,t)-h_{x_0,u_0}(y,t)\right|\\
=&\left|\inf_{z\in M}h_{z,h_{x_0,u_0}(z,t-\delta)}(x,\delta)-\inf_{z\in M}h_{z,h_{x_0,u_0}(z,t-\delta)}(y,\delta)\right|\\
\leq &\sup_{z\in M}\left|h_{z,h_{x_0,u_0}(z,t-\delta)}(x,\delta)-h_{z,h_{x_0,u_0}(z,t-\delta)}(y,\delta)\right|.
\end{align*}
Since $h_{\cdot,\cdot}(\cdot,{\delta})$ is uniformly Lipschitz on $M\times [-K,K]\times M$ with some Lipschitz constant $\kappa$, then
\[
\left|h_{x_0,u_0}(x,t)-h_{x_0,u_0}(y,t)\right|\leq \kappa\ d(x,y), \quad \forall t> 2\delta,
\]
which means
\[
\left|h_{x_0,u_0}(x,+\infty)-h_{x_0,u_0}(y,+\infty)\right|\leq \kappa\ d(x,y).
\]

Finally, we  prove that $x_0\mapsto h_{x_0,u_0}(x,+\infty)$ is Lipschitz continuous. Given $x_1,x_2\in M$,
let $\gamma:[0,d(x_1,x_2)]\to M$ be a geodesic connecting $x_1$ and $x_2$, parameterized by arclength with constant speed $\|\dot{\gamma}(s)\|_{\gamma(s)}=1$. By \cite[Lemma 3.1 and Lemma 3.2]{WWY}, we have $h_{x_1,u_0}(\gm(s),s)\ri u_0$ as $s\ri 0^+$. Combining with the uniform boundedness of $ h_{x_1,u_0}(\cdot,\cdot)$ on $M\times [\delta,+\infty)$, it follows that $h_{x_1,u_0}(\gm(s),s)$ is bounded by $C_1$ for each $s>0$. Let
	\[
	C_2:=\sup\{L(x,u,v)\ |\ x\in M,\ |u|\leq C_1,\ \|v\|_x=1\}.
	\]
	Since $\|\dot{\gamma}(s)\|_{\gamma(s)}=1$ for all $s\in[0,d(x_1,x_2)]$ and $|h_{x_1,u_0}(\gm(s),s)|\leq C_1$, we have
	 \[
	 L(\gamma(s),h_{x_1,u_0}(\gm(s),s),\dot{\gamma}(s))\leq C_2,\quad \forall s\in[0,d(x_1,x_2)].
	 \]
Let $\Delta t:=d(x_1,x_2)$. By definition, we have
\begin{align*}
h_{x_1,u_0}(x_2,\Delta t)&\leq u_0+\int_0^{\Delta t}L(\gamma(s),h_{x_1,u_0}(\gm(s),s),\dot{\gamma}(s))ds,\\
&\leq u_0+C_2 \Delta t.
\end{align*}
Moreover,
\begin{align*}
h_{x_1,u_0}(x,t+\Delta t)&\leq h_{x_2,h_{x_1,u_0}(x_2,\Delta t)}(x,t)\\
&\leq h_{x_2,u_0+C_2 \Delta t}(x,t)\\
&\leq h_{x_2,u_0}(x,t)+ C_2 \Delta t.
\end{align*}
It follows that
\[h_{x_1,u_0}(x,+\infty)\leq h_{x_2,u_0}(x,+\infty)+C_2d(x_1,x_2).\]
By exchanging the roles of $x_1$ and $x_2$, we know that $x_0\mapsto h_{x_0,u_0}(x,+\infty)$ is also Lipschitz continuous.
\end{proof}

\begin{proposition}\label{aclippinf}
	The function $(x_0,u_0,x)\mapsto \inf_{s>0}h_{x_0,u_0}(x,s)$ is uniformly Lipschitz  continuous on $M\times\R\times M$.
\end{proposition}
\begin{proof}
This proof is similar to Proposition \ref{aclipp}. Given $(x_0,x)\in M\times M$, $t>0$, by Proposition \ref{nonehh}, $u_0\mapsto h_{x_0,u_0}(x,t)$ is uniformly Lipschitz  continuous on $\R$. Moreover,
\[|\inf_{s>0}h_{x_0,u_1}(x,s)-\inf_{s>0}h_{x_0,u_2}(x,s)|\leq \sup_{s>0}|h_{x_0,u_1}(x,s)-h_{x_0,u_2}(x,s)|\leq |u_1-u_2|.\]

Next, we  prove that $x\mapsto \inf_{s>0}h_{x_0,u_0}(x,s)$ is Lipschitz continuous.  By the proof of Proposition \ref{aclipp}, $h^{\cdot,\cdot}(\cdot,{\delta})$ is uniformly Lipschitz on $M\times [-K,K]\times M$ with some Lipschitz constant $\kappa$, Namely
\[
\left|h_{x_0,u_0}(x,t)-h_{x_0,u_0}(y,t)\right|\leq \kappa\ d(x,y), \quad \forall t> 2\delta,
\]
which means
\[
\left|\inf_{s>0}h_{x_0,u_0}(x,s)-\inf_{s>0}h_{x_0,u_0}(y,s)\right|\leq \sup_{s>0}|h_{x_0,u_0}(x,s)-h_{x_0,u_0}(y,s)|\leq \kappa\ d(x,y).
\]

Finally, we  prove that $x_0\mapsto \inf_{s>0}h_{x_0,u_0}(x,s)$ is Lipschitz continuous. Given $x_1,x_2\in M$,
let $\Delta t:=d(x_1,x_2)$. By the proof of Proposition \ref{aclipp}, we have
\begin{align*}
h_{x_1,u_0}(x,t+\Delta t)\leq h_{x_2,u_0}(x,t)+ C_2 \Delta t.
\end{align*}
It follows that
\[\inf_{s>0}h_{x_1,u_0}(x,s)\leq \inf_{s>0}h_{x_2,u_0}(x,s)+C_2d(x_1,x_2).\]
By exchanging the roles of $x_1$ and $x_2$, we know that $x_0\mapsto \inf_{s>0}h_{x_0,u_0}(x,s)$ is also Lipschitz continuous.
\end{proof}

Let $(x(\cdot),u(\cdot)):\mathbb{R}\to M\times\mathbb{R}$ be a static curve. By Proposition \ref{con}, \[h_{x(s),u(s)}(x(t),+\infty):=\lim_{\tau\ri+\infty}h_{x(s),u(s)}(x(t),\tau)\] is well defined for each $s,t\in \R$. By \cite[Proposition 3.2]{WWY2}, there holds

\begin{proposition}\label{lem3.1}
	If $(x(\cdot),u(\cdot)):\R\ri M\times\R$ is static, then
	\[
	u(t)=h_{x(s),u(s)}(x(t),+\infty),\quad \forall s,\ t\in\mathbb{R}.
	\]
\end{proposition}

\subsection{Asymptotic behavior}

We only need to prove that for each $(x_0,v_0,p_0)\in \tilde{\mathcal{N}}^+$, the $\omega$-limit set  of $(x_0,v_0,p_0)$ is contained in $\tilde{\mathcal{A}}$. Based on the invariance (resp. forward invariance) of $\tilde{\mathcal{A}}_{v_-}$ (resp. $\tilde{\mathcal{N}}_{v_-}^+$),  for each $(x_0,v_0,p_0)\in \tilde{\mathcal{N}}_{v_-}^+$, $\omega$-limit set of $(x_0,v_0,p_0)$ is contained in $\tilde{\mathcal{A}}_{v_-}$.

Let $z:=(x_0,v_0,p_0)\in \tilde{\mathcal{N}}^+$. Denote $(x(t),u(t),p(t))=\Phi_t(z)$ for each $t\geq 0$. Then $(x(\cdot),u(\cdot)):[0,+\infty)\ri M\times\R$ is positively semi-static. Let $z_\omega$ be an $\omega$-limit point of the orbit passing through $z$. Denote $(x_{\omega}(t),u_{\omega}(t),p_{\omega}(t))=\Phi_t(z_\omega)$. By Theorem \ref{88996},  $(x_{\omega}(t),u_{\omega}(t))$ is semi-static, namely for each $t_1\geq t_2$,
\begin{equation}\label{444}
u_{\omega}(t_1)=\inf_{s>0}h_{x_{\omega}(t_2),u_{\omega}(t_2)}(x_{\omega}(t_1),s).
\end{equation}
It remains to show that for each $t_1<t_2$, (\ref{444}) still holds.  One can find $(x(s_n),u(s_n))$ and $(x(t_n),u(t_n))$ with
\[s_n<s_n+(t_2-t_1)=:t_n<s_{n+1},\quad s_n,t_n\rightarrow+\infty,\quad as \ n\rightarrow\infty,\]
and \[x(s_n)\rightarrow x_{\omega}(t_1),\quad u(s_n)\rightarrow u_{\omega}(t_1),\quad x(t_n)\rightarrow x_{\omega}(t_2),\quad u(t_n)\rightarrow u_{\omega}(t_2).\]
Since $(x(\cdot),u(\cdot)):[0,+\infty)\ri M\times\R$ is positively semi-static, we have
\[u(s_{n+1})=\inf_{s>0}h_{x(t_n),u(t_n)}(x(s_{n+1}),s).\]
Take a limiting passage as $n\rightarrow\infty$, then (\ref{444}) can be verified by Proposition \ref{aclippinf}.

This completes the proof of Theorem \ref{namss}(2).

\subsection{Closedness of $\tilde{\A}$}

	Let $\{(x_n,u_n,p_n)\}_n\subseteq \tilde{\mathcal{A}}$ with
	$(x_n,u_n,p_n)\to (x_0,u_0,p_0)$ as $n\to\infty$. It is sufficient to show that $(x_0,u_0,p_0)\in \tilde{\mathcal{A}}$.
 Let
 \[(x_0(t),u_0(t),p_0(t)):=\Phi_t(x_0,u_0,p_0).\]
 Note that $\tilde{\mathcal{A}}\subseteq \tilde{\N}$. It follows from Proposition \ref{ooo} and the continuity and differentiability of solutions of ordinary differential equations with respect to initial values,	 $(x_0(t),u_0(t),p_0(t))$ exists for all $t\in \mathbb{R}$. Thus, we only need to prove $(x_0(\cdot),u_0(\cdot)):\R\ri M\times\R$ is static.

	For each $n\in \mathbb{N}$, let $(x_n(t),u_n(t),p_n(t))$ denote the solution of equations (\ref{c}) with initial value $(x_n,u_n,p_n)$. For each $t_1$, $t_2\in\mathbb{R}$, take $K>0$ such that $t_1,t_2\in[-K,K]$. By a similar argument as the global characterizations of Ma\~{n}\'{e} sets,
	we have \[\inf_{s>0}h_{x_0(t_1),u_0(t_1)}(x_0(t_2),s)\geq u_0(t_2).\]

On the other hand, by Proposition \ref{lem3.1}, $h_{x_n(t_1),u_n(t_1)}(x_n(t_2),+\infty)= u_n(t_2)$. Combining with Proposition \ref{aclipp}, we have
\[h_{x_0(t_1),u_0(t_1)}(x_0(t_2),+\infty)= u_0(t_2),\]
which implies $\inf_{s>0}h_{x_0(t_1),u_0(t_1)}(x_0(t_2),s)\leq u_0(t_2)$.

	 Since $(x_n(t),u_n(t))$ is globally minimizing, then it is clear to see that $(x_0(t),u_0(t))$ is also a globally minimizing curve. Therefore, $(x_0(\cdot),u_0(\cdot)):\R\ri M\times\R$ is static.

\subsection{Inclusion relations}The
 Mather set is defined by
\[\tilde{\mathcal{M}}=\mathrm{cl}\left(\bigcup_{\mu\in \mathfrak{M}}\text{supp}(\mu)\right),\]
where $\mathfrak{M}$ denotes the set of Mather measures.
By definition,  $\tilde{\A}\subseteq \tilde{\N}$. It remains to show
\[\tilde{\mathcal{M}}\subseteq \tilde{\A}.\]
By definition, $\tilde{\A}$ is closed. Note that Mather measures are  invariant Borel probabilities measures. Given $\mu\in \mathfrak{M}$. By the Poincar\'{e} recurrence theorem,  one can find a set $A\subseteq T^*M\times\R$ of total $\mu$-measure such that if $(x_0,u_0,p_0)\in A$, then there exist $\{t_m\}_{m\in \mathbb{N}}$  such that
\[d\left((x_0,u_0,p_0),\Phi_{t_m}(x_0,u_0,p_0)\right)\ri 0\quad \text{as}\ \ t_m\ri +\infty,\]
where $d(\cdot,\cdot)$ denotes the distance induced by the Riemannian metric on $T^*M\times\R$.
Let
\[(x(t),u(t),p(t)):=\Phi_t(x_0,u_0,p_0),\quad \forall t\in \R.\]
Since $\tilde{\A}$ is closed and $A$ is dense in supp$(\mu)$, then we only need to show that  $(x(\cdot),u(\cdot)):\mathbb{R}\to M\times\mathbb{R}$ is a static curve.
By the definition of $\tilde{\mathcal{M}}$, $(x(\cdot),u(\cdot)):\mathbb{R}\to M\times\mathbb{R}$ is semi-static.  Let $p(t):=\frac{\partial L}{\partial \dot{x}}(x(t),u(t),\dot{x}(t))$. By assumption,
\[d\left((x(t_2),u(t_2),p(t_2)),\Phi_{t_m}(x(t_2),u(t_2),p(t_2))\right)\ri 0\quad \text{as}\ \ t_m\ri +\infty.\]
We only need to prove 	
\begin{equation}\label{uhxtt}
	u(t_2)=\inf_{s>0}h_{x(t_1),u(t_1)}(x(t_2),s).
\end{equation}

 Without loss of generality, we assume $t_1\geq t_2$. Since $(x(\cdot),u(\cdot)):\mathbb{R}\to M\times\mathbb{R}$ is semi-static, we have
	\begin{equation}\label{uhxtt77}
	u(t_1)=\inf_{s>0}h_{x(t_2),u(t_2)}(x(t_1),s).
\end{equation}
 Let $\Delta:=t_1-t_2$. By (\ref{uhxtt77}), if $t_m>\Delta$, we have
\[u(t_2+t_m)=\inf_{s>0}h_{x(t_1),u(t_1)}(x(t_2+t_m),s).\]
Let $t_m\ri +\infty$. It follows from Proposition \ref{aclippinf} that
\[u(t_2)=\inf_{s>0}h_{x(t_1),u(t_1)}(x(t_2),s),\]
which together with (\ref{uhxtt77}) implies
 $(x(\cdot),u(\cdot)):\mathbb{R}\to M\times\mathbb{R}$ is a static curve.

 This completes the proof of Theorem \ref{namss}(3).

\section{More on Aubry sets}\label{maset}

This part is devoted to proving Theorem \ref{crogrpccc} and Theorem \ref{grpaord}.

\subsection{Comparison property}
By contradiction, we assume that there exists $x_0\in M$ such that $v_-(x_0)<u_-(x_0)$. Let $\gm:(-\infty,0]\rightarrow M$ be a $(v_-, L, 0)$-calibrated curve with $\gm(0)=x_0$. Let
\[F(s):=u_-(\gm(s))-v_-(\gm(s)).\]
By assumption, we have $F(0)>0$. Note that $F(s)$ is continuous, we have a dichotomy as follows.
\begin{itemize}
\item [(1)] There exists $s_0<0$ such that $F(s_0)=0$ and $F(s)>0$ for any $s\in (s_0,0]$;
\item [(2)] $F(s)>0$ for all $s\in (-\infty, 0]$.
\end{itemize}

For Item (1),  we have
\[v_-(\gm(s))-v_-(\gm(s_0))=\int_{s_0}^sL(\gm(\tau),v_-(\gm(\tau)),\dot{\gm}(\tau))d\tau,\]
\[u_-(\gm(s))-u_-(\gm(s_0))\leq\int_{s_0}^sL(\gm(\tau),u_-(\gm(\tau)),\dot{\gm}(\tau))d\tau.\]
It follows that
\[F(s)\leq \lambda\int_{s_0}^sF(\tau)d\tau,\]
which implies $F(s)=0$ for any $s\in [s_0,0]$. This contradicts $F(0)>0$. It means $u_-\leq v_-$ on $M$.

It remains to prove this theorem for Item (2). Let  $u_0:=v_-(x_0)$, $p_0:=\frac{\partial L}{\partial \dot{x}}(x_0,u_0,\dot{\gm}(0)_-)$, where $\dot{\gm}(0)_-$ denotes the left derivative of $\gm(t)$ at $t=0$. According to Proposition \ref{pr1188}, let $x(t):=\gm(t)$, $u(t):=v_-(\gm(t)))$, then  $(x(\cdot),u(\cdot)):(-\infty,0]\ri M\times\R$ is negatively semi-static. Let $\alpha(x_0,u_0,p_0)$ be the $\alpha$-limit set of $(x_0,u_0,p_0)$. By Theorem \ref{namss}(2), we have
\[\alpha(x_0,u_0,p_0)\subset \tilde{\A}_{v_-}.\]
By definition, one can find a subsequence $\{s_n\}_{n\in \mathbb{N}}$ such that $\gm(s_n)\ri \bar{x}\in {\A}_{v_-}$ as $s_n\ri -\infty$. Since $u_-\leq v_-$ on $\A_{v_-}$, then $u_-(\bar{x})\leq v_-(\bar{x})$, which yields $\lim_{s_n\ri-\infty}F(s_n)\leq 0$.
Note that
\[v_-(x_0)-v_-(\gm(s))=\int_{s}^0L(\gm(\tau),v_-(\gm(\tau)),\dot{\gm}(\tau))d\tau,\]
\[u_-(x_0)-u_-(\gm(s))\leq\int_{s}^0L(\gm(\tau),u_-(\gm(\tau)),\dot{\gm}(\tau))d\tau.\]
By (H3), there holds for each $s\in (-\infty,0]$,
	\begin{equation}\label{crduvuvv}
u_-(x_0)-v_-(x_0)\leq u_-(\gm(s))-v_-(\gm(s)).
\end{equation}
In particular, (\ref{crduvuvv}) still holds for $\{s_n\}_{n\in \mathbb{N}}$. Let $s_n\ri -\infty$, we have $u_-(x_0)-v_-(x_0)\leq 0$, which contradicts the assumption $v_-(x_0)<u_-(x_0)$. This completes the proof of Theorem \ref{crogrpccc}.

\begin{remark}\label{macom}
By the definition of the Mather set, the $\alpha$-limit set of $(x_0,u_0,p_0)$ carries Mather measures. Consequently, there exists a sequence $t_n\to +\infty$ such that $\gm(t_n)\to \tilde{x}\in \mathcal{M}_{v_-}$, where
\[\mathcal{M}_{v_-}=\pi\big(\tilde{\mathcal{M}}\cap G_{v_-}\big).\]
 Given  $u_-,v_-\in \cS_-$. Following from the same argument as the proof of Comparison property,  if $u_-\leq v_-$ on $\mathcal{M}_{v_-}$, then $u_-\leq v_-$ on $M$. Moreover, if $u_-=v_-$ on $\mathcal{M}_{u_-}\cup\mathcal{M}_{v_-}$, then $u_-=v_-$ on $M$.
\end{remark}

\subsection{Static curves and positively semi-static curves}
As a preparation to prove the injectivity of $\Pi:\rho(\bar{\mathcal{A}})\to M$, we show that for certain minimizing orbits $(x(\cdot),p(\cdot),u(\cdot)):\R\ri T^*M\times\R$, $p(t)$ is uniquely determined by $(x(t),u(t))$ for each $t\in\R$.

\begin{lemma}\label{lem3.2}
	If $(x,u,p_1)\in \tilde{\mathcal{A}}$, $(x,u,p_2)\in \tilde{\N}^+$,   then $p_1=p_2$.
\end{lemma}

\begin{proof}
For each $t\in \R$, let \[(x_1(t),u_1(t),p_1(t)):=\Phi_t(x,u,p_1).\]For each $t\geq 0$, let  \[(x_2(t),u_2(t),p_2(t)):=\Phi_t(x,u,p_2).\]

Since  $(x,u,p_2)\in \tilde{\N}^+$, then $(x_2(t),u_2(t))$ is positively semi-static.
Fix $\delta>0$, by the Markov property, we have
\[h_{x_1(-\delta),u_1(-\delta)}(x_2(\delta),2\delta)=\inf_{y\in M}h_{y,h_{x_1(-\delta),u_1(-\delta)}(y,\delta)}(x_2(\delta),\delta).\]
Note that
\[h_{x_1(-\delta),u_1(-\delta)}(x,\delta)=u.\]
	It follows that
\[h_{x_1(-\delta),u_1(-\delta)}(x_2(\delta),2\delta)\leq h_{x,u}(x_2(\delta),\delta).\]
We assert that the equality holds. If the assertion is true, then by Proposition \ref{Mar-new}, the curve defined by	
	\[
	\gamma(\sigma):=\left\{\begin{array}{ll}
	x_1(\sigma-\delta), \quad\sigma\in[0,\delta],\\
	x_2(\sigma-\delta),\,\quad\sigma\in[\delta,2\delta],
	\end{array}\right.
	\]
	is a minimizer of $h_{x_1(-\delta),u_1(-\delta)}(x_2(\delta),2\delta)$ and thus is of class $C^1$. Thus,
\[p_1=\frac{\partial L}{\partial \dot{x}}(x,u,\dot{\gm}(0))=p_2.\]

It remains to verify the assertion. By contradiction, we assume that  there exists $\Delta>0$ such that	
	\[
	h_{x_1(-\delta),u_1(-\delta)}(x_2(\delta),2\delta)=h_{x,u}(x_2(\delta),\delta)-\Delta.
	\]

By assumption,  for each $\eps>0$, one can find $s_0>0$ such that
\[|h_{x,u}(x_1(-\delta),s_0)-u_1(-\delta)|\leq \eps.\]
From Proposition \ref{nonehh},
\[|h_{x_1(-\delta),h_{x,u}(x_1(-\delta),s_0)}(x_2(\delta),2\delta)-h_{x_1(-\delta),u_1(-\delta)}(x_2(\delta),2\delta)|\leq \eps.\]
It follows that
\begin{align*}
u_2(\delta)&=\inf_{\tau>0}h_{x,u}(x_2(\delta),\tau)\\
&\leq h_{x,u}(x_2(\delta),s_0+2\delta)\\
&\leq h_{x_1(-\delta),h_{x,u}(x_1(-\delta),s_0)}(x_2(\delta),2\delta)\\
&\leq h_{x_1(-\delta),u_1(-\delta)}(x_2(\delta),2\delta)+\eps\\
&=h_{x,u}(x_2(\delta),\delta)-\Delta+\eps.
\end{align*}
Note that $\Delta$ is independent of $\eps$, we have
	\[
	u_2(\delta)\leq h_{x,u}(x_2(\delta),\delta)-\frac{\Delta}{2}=u_2(\delta)-\frac{\Delta}{2},
	\]
	which is a contradiction.
\end{proof}

\subsection{Graph property of the Aubry set}

Lemma \ref{lem3.2} implies that $p$-component is uniquely determined by $(x,u)\in M\times\R$. It follows that the standard projection from $\tilde{\mathcal{A}}$ to $M\times\R$  is injective. In the following, we move one more step to show that $u$-component can be neglected if we consider the dual Aubry set $\bar{\mathcal{A}}$ in $TM\times\R$. Namely, $\dot{x}$-component is uniquely determined by $x\in M$. Let us recall that $\rho:TM\times\R\ri TM$ and $\Pi:TM\ri M$ denote the standard projections.

\begin{lemma}\label{kkkee}
	Let $(x(\cdot),u(\cdot)):\R\ri M\times\R$ be a static curve. Let $u_0:=u(0)$, $x_0:=x(0)$ and
	\[v(t):=\inf_{\tau>0}h_{x_0,v_0}(x(t),\tau).\]
	If $v_0\geq u_0$, and $h_{x_0,v_0}(x_0,+\infty)=v_0$, then  $(x(\cdot),v(\cdot)):\R\ri M\times\R$ is also static.
\end{lemma}
The proof of Lemma \ref{kkkee} is standard, which is given in Appendix \ref{prkkkee}.
Lemma \ref{kkkee} implies
	$\Pi:\rho(\bar{\mathcal{A}})\to M$ is  injective.
In fact, let $(x_i(\cdot),u_i(\cdot)):\R\ri M\times\R$ be static, where $i=1,2$. We  prove that if $x_1(0)=x_2(0)=x_0$, then $x_1(t)=x_2(t)$ for all $t>0$, which implies $\dot{x}_1(t)=\dot{x}_2(t)$ for each $t\geq 0$. Let $u_1:=u_1(0)$, $u_2:=u_2(0)$ and $\bar{u}_2(t)=\inf_{s>0}h_{x_0,u_2}(x_1(t),s)$. Then
\[\bar{u}_2(0)=\inf_{s>0}h_{x_0,u_2}(x_1(0),s)=\inf_{s>0}h_{x_2(0),u_2(0)}(x_2(0),s)=u_2.\]

We assume $u_2\geq u_1$, the other case is similar. By Proposition \ref{lem3.1}, $h_{x_0,u_2}(x_0,+\infty)=u_2$.
	By Lemma \ref{kkkee}, $(x_1(t),\bar{u}_2(t))$ is static, then there is a solution of equations (\ref{c}) denoted by $(x_1(t),\bar{u}_2(t),\bar{p}_2(t))$ with
\[\bar{p}_2(t):=\frac{\partial L}{\partial \dot{x}}(x_1(t),\bar{u}_2(t),\dot{x}_1(t)),\]
 and $(x_0,u_2,\bar{p}_2(0))\in \tilde{\mathcal{A}}$. Since $(x_2(t),u_2(t))$ is static, there is a solution of equations (\ref{c}) denoted by $(x_2(t),u_2(t),p_2(t))$ with
\[{p}_2(t):=\frac{\partial L}{\partial \dot{x}}(x_2(t),{u}_2(t),\dot{x}_2(t)),\]
 and  $(x_0,u_2,p_2(0))\in \tilde{\mathcal{A}}$.
	By Lemma \ref{lem3.2}, we have $\bar{p}_2(0)=p_2(0)$, which implies that  for all $t\geq 0$, \[(x_1(t),\bar{u}_2(t),\bar{p}_2(t))=(x_2(t),u_2(t),p_2(t)).\] In particular, $\dot{x}_1(0)=\dot{x}_2(0)$. Namely,  	$\Pi:\rho(\bar{\mathcal{A}})\to M$ is  injective.

\subsection{Partially ordered relation}

Given $\bar{x}\in \mathcal{A}_{u_-}$. We  prove $\bar{x}\in \mathcal{A}_{v_-}$. By definition, one can find a static curve  $(x(\cdot),u(\cdot)):\R\ri M\times\R$ such that $x(0)=\bar{x}$ and  $u(t)=u_-(x(t))$ for each $t\in \R$. Let $v(t):=v_-(x(t))$ for each $t\in \R$. It suffices to show that $(x(\cdot),v(\cdot)):\R\ri M\times\R$ is also a static curve.

Since $u_-\leq v_-$ on $\A_{u_-}$, then $u(0)\leq v(0)$. It follows that for each $s>0$, we have
		\begin{equation}\label{ww44}
h_{x(0),v(0)}(x(0),s)-h_{x(0),u(0)}(x(0),s)\leq v(0)-u(0).
	\end{equation}
By Lemma \ref{lem3.1}, 	$h_{x(0),u(0)}(x(0),+\infty)= u(0)$, then
	$h_{x(0),v(0)}(x(0),+\infty)\leq v(0)$. Since $v_-\in \mathcal{S}_-$, then for each $t>0$,
		\begin{equation}\label{ww45}
	v(0)=v_-(x(0))=T_t^-v_-(x(0))\leq h_{x(0),v(0)}(x(0),t),
	\end{equation}	
	which gives rise to $h_{x(0),v(0)}(x(0),+\infty)= v(0)$.

Let $\bar{v}(t):=\inf_{\tau>0}h_{x(0),v(0)}(x(t),\tau)$. By (\ref{ww44}), $\bar{v}(0)\leq v(0)$, which combining with (\ref{ww45}) yields $\bar{v}(0)=v(0)$. By Lemma \ref{kkkee}, $(x(t),\bar{v}(t))$ is static and
	\begin{equation}\label{8ii}
	\bar{v}(t)-\bar{v}(0)=\bar{v}(t)-{v}(0)=u(t)-u(0).
	\end{equation}

 It remains to show that $\bar{v}(t)=v(t)$.
	Since $v_-\in \mathcal{S}_-$, then $v(t)=v_-(x(t))\leq \bar{v}(t)$.

We  prove $v(t)\geq  \bar{v}(t)$ in the following.
By contradiction, we assume that there exists $t_0\in \R$ such that $v(t_0)<\bar{v}(t_0)$. Note that $u_-\leq v_-$ on $\A_{u_-}$, we have $u(t_0)\leq v(t_0)$.  For each $s>0$, we have
		\begin{equation}\label{ww54}
h_{x(t_0),v(t_0)}(x(t_0),s)-h_{x(t_0),u(t_0)}(x(t_0),s)\leq v(t_0)-u(t_0).
	\end{equation}
 Since $h_{x(t_0),u(t_0)}(x(t_0),+\infty)=u(t_0)$, we have
	$h_{x(t_0),v(t_0)}(x(t_0),+\infty)\leq v(t_0)$.  On the other hand, since $v_-\in \mathcal{S}_-$, then for each $t>0$,
			\begin{equation}\label{ww55}
	v(t_0)=v_-(x(t_0))=T_t^-v_-(x(t_0))\leq h_{x(t_0),v(t_0)}(x(t_0),t),
	\end{equation}
	which gives rise to $h_{x(t_0),v(t_0)}(x(t_0),+\infty)= v(t_0)$. Let \[\tilde{v}(t):=\inf_{\tau>0}h_{x(t_0),v(t_0)}(x(t),\tau).\] By (\ref{ww54}) and (\ref{ww55}), we have $\tilde{v}(t_0)=v(t_0)$. By Lemma \ref{kkkee}, $(x(t),\tilde{v}(t))$ is also static and $\tilde{v}(t)-\tilde{v}(t_0)=u(t)-u(t_0)$ for each $t\in \R$. Thus, we have
	\begin{equation}\label{9ii}
	\tilde{v}(0)-\tilde{v}(t_0)=u(0)-u(t_0).
	\end{equation}
By (\ref{8ii}), we have
	\begin{equation}\label{9ixxi}
	\bar{v}(t_0)-{v}(0)=u(t_0)-u(0),
	\end{equation}
which together with (\ref{9ii}) and $\tilde{v}(t_0)=v(t_0)$
 implies
\begin{equation}\label{7ii}
\bar{v}(t_0)-{v}(t_0)=v(0)-\tilde{v}(0).
\end{equation}
Note that $\bar{v}(t_0)>v(t_0)$, then $v(0)-\tilde{v}(0)>0$. On the other hand, using $v_-\in \mathcal{S}_-$ again,  for each $\tau>0$,
	\[
	v(t)=v_-(x(t))=T_{\tau}^-v_-(x(t))\leq h_{x(t_0),v(t_0)}(x(t),\tau),
	\]
which yields $v(t)\leq \tilde{v}(t)$ for each $t\in \R$. In particular, we have  $v(0)-\tilde{v}(0)\leq 0$, which contradicts $v(0)-\tilde{v}(0)>0$.

This completes the proof of Theorem \ref{grpaord}.


\section{Strongly static set}\label{ssaset}
In this part, we  prove Theorem \ref{conterstatic} and Proposition \ref{exppp}.
\begin{proposition}\label{pr4.10}
	Let $x_1$, $x_2\in M$ and $u_1$, $u_2\in \mathbb{R}$.
	If $u_2=\sup_{s>0}h^{x_1,u_1}(x_2,s)$, then $u_1=\inf_{s>0}h_{x_2,u_2}(x_1,s)$.
\end{proposition}
\begin{proof}
	For each $s>0$, we have $u_2\geq h^{x_1,u_1}(x_2,s)$, then $u_1\leq h_{x_2,u_2}(x_1,s)$, which means
	\[
	u_1\leq \inf_{s>0}h_{x_2,u_2}(x_1,s).
	\]
	We assume by contradiction that there exists $\delta>0$ such that for each $s>0$,
	$h_{x_2,u_2}(x_1,s)\geq u_1+\delta$. Since
	\[
	h_{x_2,u_2}(x_1,s)-h_{x_2,u_2-\frac{\delta}{2}}(x_1,s)\leq u_2-(u_2-\frac{\delta}{2})=\frac{\delta}{2},
	\]
	then we have
	\[
	h_{x_2,u_2-\frac{\delta}{2}}(x_1,s)\geq u_1+\frac{\delta}{2},
	\]
	which implies
	\[
	u_2-\frac{\delta}{2}\geq h^{x_1,u_1+\frac{\delta}{2}}(x_2,s)\geq h^{x_1,u_1}(x_2,s).
	\]
	It contradicts $u_2=\sup_{s>0}h^{x_1,u_1}(x_2,s)$.
\end{proof}

By Proposition \ref{pr4.10}, we have $\tilde{\mathcal{S}}_s\subseteq \tilde{\mathcal{A}}$. Then, we  show $\tilde{\mathcal{S}}_s=\tilde{\mathcal{A}}$ if $H$ is independent of $u$.
In classical cases, the action function is defined by
\[h^s(x_0,x):=\inf_{\gm}\int_0^sL(\gm(\tau),\dot{\gm}(\tau))d\tau,\]
where the infimums are taken among the absolutely continuous curves $\gamma:[0,s]\rightarrow M$ with $\gm(0)=x_0$ and $\gm(s)=x$. In terms of the notations in contact cases,
if $H(x,u,p)$ is independent of $u$, then
\begin{equation}\label{hchch}
h_{x_0,u_0}(x,s)=u_0+h^s(x_0,x),\quad h^{x_0,u_0}(x,s)=u_0-h^s(x,x_0).
\end{equation}
In view of Proposition \ref{pr4.10}, we only need to prove if $u_1=\inf_{s>0}h_{x_2,u_2}(x_1,s)$, then $u_2=\sup_{s>0}h^{x_1,u_1}(x_2,s)$ for each $x_1$, $x_2\in M$ and $u_1$, $u_2\in \mathbb{R}$. That is a direct consequence of (\ref{hchch}).

 We complete the proof of Theorem \ref{conterstatic}.

\vspace{1em}

In the remaining part of this section, we  prove Proposition \ref{exppp}.
The contact Hamilton equation reads
\begin{align}\label{xjja}
\left\{
        \begin{array}{l}
        \dot{x}=p+V(x),\\
        \dot{p}=-pV'(x)-\lambda p,\\
        \dot{u}=p(p+V(x))-H(x,u,p).
         \end{array}
         \right.
\end{align}
Obviously, $u_-\equiv 0$ is the unique viscosity  solution of $H(x,u,Du)=0$. Note that if $x(t)$ satisfies $\dot{x}=V(x)$ for each $t\in \R$, then $(x(t),0,0)$ satisfies (\ref{xjja}). It follows that
\[\tilde{\A}=\cap_{t\leq 0}\Phi_t(G_{u_-})=G_{u_-}=\left\{(x,0,0)\ |\ x\in \mathbb{T}\right\}.\]

 We assert that
\[\tilde{\mathcal{S}}_s=\{(x_1,0,0),\ (x_2,0,0)\}.\]
By Legendre transformation,
\[L(x,u,v):=-\lambda u+\frac{1}{2}|v-V(x)|^2,\quad x\in \mathbb{T}.\]
 By definition, a direct calculation shows
\begin{align}\label{2-4hhh}
h^{x_0,u_0}(x,t)=e^{\lambda t}u_0-\inf_{\substack{\gamma(t)=x_0 \\  \gamma(0)=x } }\int_0^te^{\lambda s}\frac{1}{2}|\dot{\gm}(s)-V(\gm(s))|^2ds.
\end{align}
Then we have
\[0=\sup_{\tau>0}h^{x_i,0}(x_i,\tau),\quad i=1,2.\]
It follows that $\{(x_1,0,0), (x_2,0,0)\}\subseteq \tilde{\mathcal{S}}_s$.

\vskip 0.2cm

Let $[x_1,y_0)$ be a fundamental domain of $\mathbb{T}$. In order to complete the proof, it remains to show that
\[[x_1,y_0)\backslash \{x_1,x_2\}\cap {\mathcal{S}}_s=\emptyset.\]
Let
\[I:=[x_1,y_0),\ I_1:=(x_1,x_2),\ I_2:=(x_2,y_0).\]
Note that $\tilde{\mathcal{S}}_s$ is flow-invariant. Then we need to exclude the following three cases:
\begin{itemize}
\item [(1)] ${\mathcal{S}}_s=\{x_1,x_2\}\cup I_1$,
\item [(2)] ${\mathcal{S}}_s=\{x_1,x_2\}\cup I_2$,
\item [(3)] ${\mathcal{S}}_s=I$.
\end{itemize}
We  prove that Case (1) does not happen. The other cases are similar.

Let  $(x(\cdot),u(\cdot)):\R\ri I_1\times\R$  be a static curve. If Case (1) holds, then $(x(\cdot),u(\cdot)):\R\ri I_1\times\R$ is also strongly static. By \cite[Lemma 4.8]{WWY2}, $u(t)=u_-(x(t))=0$ for all $t\in \R$. Given $\eps_0>0$ small enough, take $x(t_1),x(t_2)\in [\frac{x_1+x_2}{2}-\eps_0,\frac{x_1+x_2}{2}+\eps_0]$ and $x(t_1)<x(t_2)$.
By definition,   there holds
	\begin{equation*}
	u(t_2)=\sup_{\tau>0}h^{x(t_1),u(t_1)}(x(t_2),\tau).
\end{equation*}
Since $x(t_1)\neq x(t_2)$,  by (H2), there exists $\delta>0$ such that
	\begin{equation*}
	u(t_2)=\sup_{\tau>\delta}h^{x(t_1),u(t_1)}(x(t_2),\tau).
\end{equation*}
 It follows from $u(t_1)=u(t_2)=0$ that
\begin{equation}\label{itagam}
0=\inf_{\tau>\delta}\inf_{\xi}\int_{0}^{\tau}e^{\lambda s}\frac{1}{2}|\dot{\xi}(s)-V(\xi(s))|^2ds,
\end{equation}
where $\xi$ is taken among all Lipschitz continuous curves with $\xi(0)=x(t_2)$ and $\xi(\tau)=x(t_1)$.
By the variational principle,  given $\tau>\delta$, the infimum is achieved at $\gm:[0,\tau]\ri M$, which is of class $C^1$. Let
\begin{align*}
v(s):=h^{x(t_1),0}(\gamma(s),\tau-s),\,\,\, p(s):=\frac{\partial L}{\partial \dot{x}}(\gamma(s),v(s),\dot{\gamma}(s)),
\end{align*}
Then $(\gm(s),v(s),p(s))$  satisfies equations \eqref{xjja} with
\begin{align*}
\gm(0)=x(t_2), \quad \gm(\tau)=x(t_1), \quad \lim_{s\rightarrow \tau^-}v(s)=0,
\end{align*}
which implies
\begin{equation}\label{psvv}
p(s)=\dot{\gm}(s)-V(\gm(s)),\quad \forall s\in (0,\tau).
\end{equation}
If there exists $s_0\in [0,\tau)$ such that $\dot{\gm}(s_0)<0$ and $\gm(s_0)\in [x(t_1),x(t_2)]$.  Based on the construction of $V(x)$, there exists $C_2>0$  independent of $\tau$ such that $V(x)>C_2>0$ for all $x\in [x(t_1),x(t_2)]\subset [\frac{x_1+x_2}{2}-\eps_0,\frac{x_1+x_2}{2}+\eps_0]$. By (\ref{psvv}),
\[p(s_0)=\dot{\gm}(s_0)-V(\gm(s_0))<-V(\gm(s_0))<-C_2.\]
Otherwise, one can find $s_1\in [0,\tau)$ such that $\dot{\gm}(s_1)>0$ and $\gm(s_1)\in [\frac{x_2+y_0}{2}-\eps_0,\frac{x_2+y_0}{2}+\eps_0]$.  Based on the construction of $V(x)$, there exists $C_3>0$  independent of $\tau$ such that $V(x)<-C_3<0$ for all $x\in [\frac{x_2+y_0}{2}-\eps_0,\frac{x_2+y_0}{2}+\eps_0]$. By (\ref{psvv}),
\[p(s_1)=\dot{\gm}(s_1)-V(\gm(s_1))>-V(\gm(s_1))>C_3>0.\]
It follows that for each $\tau>\delta$,
\[\inf_{\xi}\int_{0}^{\tau}e^{\lambda s}\frac{1}{2}|\dot{\xi}(s)-V(\xi(s))|^2ds=\int_{0}^{\tau}e^{\lambda s}\frac{1}{2}|\dot{\gm}(s)-V(\gm(s))|^2ds>C_4>0,\]
where $C_4$ is a constant independent of $\tau$. This contradicts (\ref{itagam}).

\section{Applications}\label{appl}
In this part, we prove Proposition \ref{exammmmpp} and Proposition \ref{exvnish}.

\subsection{ Proof of Proposition \ref{exammmmpp}}

\begin{lemma}\label{crogrp}
Given  $u_-,v_-\in \cS_-$, let \[w_-(x):=\min_{x\in M}\{u_-(x),v_-(x)\}.\] Then $w_-\in \cS_-$ and  both $u_-$ and $v_-$ are of class $C^{1,1}$ on $\A_{w_-}$. Moreover,  for each $x\in \A_{w_-}$,
\[\frac{\partial H}{\partial p}(x,u_-(x),Du_-(x))=\frac{\partial H}{\partial p}(x,v_-(x),Dv_-(x)).\]  In particular, if
\begin{equation*}
H(x,u,p):=f(x,u)+h(x,p),
\end{equation*}then
$Du_-=Dv_-$ on $\A_{w_-}$.
\end{lemma}
\begin{proof}
Let $u_-$ and $v_-\in \mathcal{S}_-$. Denote
\[w_-(x):=\min_{x\in M}\{u_-(x),v_-(x)\}.\]
It follows from Lemma \ref{exyi} that $w_-\in \cS_-$. Then $S\neq \emptyset$. Since  $w\leq u_-$ and $w\leq v_-$ on $M$, by Theorem \ref{grpaord}, we have $\A_{w_-}\subseteq \A_{u_-}\cap\A_{v_-}$.
Thus,  both $u_-$ and $v_-$ are of class $C^{1,1}$ on $\A_{w_-}$.

Note that
\[\tilde{\A}_{w_-}\subseteq \tilde{\mathcal{A}}.\]
By Theorem \ref{grappff}, for each $x\in \A_{w_-}$,
\[\frac{\partial H}{\partial p}(x,u_-(x),Du_-(x))=\frac{\partial H}{\partial p}(x,v_-(x),Dv_-(x)).\]
This completes the proof.
\end{proof}

\medskip
We prove Proposition \ref{exammmmpp}. First of all, we assert that for each $x_0\in [-\frac{1}{2},\frac{1}{2})$, $(x(t),u(t))=(x_0,0)$ is a static curve.

In fact, it is easy to verify that $(x(t),u(t),p(t))=(x_0,0,0)$ satisfies (\ref{c}). Fixing $s>0$, by the minimality of $h_{x_0,0}(x_0,s)$, we have $h_{x_0,0}(x_0,s)\leq u(s)=0$. On the other hand,  for each $s>0$, $0=u_-(x_0)=T_s^-u_-(x_0)\leq h_{x_0,0}(x_0,s)$, then we have $\inf_{s>0}h_{x_0,0}(x_0,s)=0$.  Thus, we have
 \[ \left\{(x,0,0)\ \big|\forall\ x\in [-\frac{1}{2},\frac{1}{2})\right\}=\tilde{\mathcal{A}}_{u_-}\subseteq \tilde{\mathcal{A}}.\]

Second, for each $u_i<0$, $(x(t),u(t))=(0,u_i)$ is also a static curve. In fact, it is easy to verify that $(x(t),u(t),p(t))=(0,u_i,0)$ satisfies (\ref{c}). In addition, since $u_i<0$, then $h_{0,0}(0,s)-h_{0,u_i}(0,s)\leq 0-u_i$ for each $s>0$. Hence, we have $h_{0,u_i}(0,s)\geq u_i$. Besides, by the minimality, it is clear that $h_{0,u_i}(0,s)\leq u(s)=u_i$. Then for each $s>0$, there holds $h_{0,u_i}(0,s)=u_i$. It follows that $\inf_{s>0}h_{0,u_i}(0,s)=u_i$.

Moreover, we have
\[\tilde{\A}_{u_-}\bigcup\big(\cup_{i\in I}\left\{(0,u_i,0)\ |\ u_i<0\right\}\big)\subseteq \tilde{\mathcal{A}}.\]
By Theorem \ref{namss}, for each $u_i<0$, one can find $v_i\in \mathcal{S}_-$ such that $v_i(0)=u_i$. Hence, there exist an uncountable family of nontrivial viscosity solutions.

Note that $\mathcal{A}_{u_-}=\mathbb{T}=\A$. It remains to show  $\mathcal{A}_{v_i}\subsetneqq \mathbb{T}$ for each $i\in I$. We conclude it by two steps.

\medskip
\noindent{\bf Step 1.} We prove $v_i(x)\leq 0$ for each $i\in I$ and $x\in [-\frac{1}{2},\frac{1}{2})$. Let $v_i(y_0)=\max_{y\in [-\frac{1}{2},\frac{1}{2})}v_i(y)$. Note that $v_i(0)=u_i<0$. By contradiction, we assume $v_i(y_0)>0$, then $y_0\in [-\frac{1}{2},\frac{1}{2})\backslash\{0\}$. There exists $\varepsilon>0$ such that $v_i(y)>0$ for each $y\in (y_0-\varepsilon,y_0+\varepsilon)$. Note that $v_i$ is Lipschitz continuous, one can find $y_1\in (y_0-\varepsilon,y_0+\varepsilon)$ and $y_1\neq 0$ such that $v_i$ is differentiable at $y_1$. Then we have
\[\frac{1}{2}|Dv_i(y_1)|^2+f(y_1)v_i(y_1)=0,\]
which contradicts $v_i(y_1)>0$.

\medskip
\noindent{\bf Step 2.}
We prove $\mathcal{A}_{v_i}\subsetneqq \mathbb{T}=\A_{u_-}$. By contradiction, we assume \[\A_{v_i}=\mathbb{T}.\]
Since $v_i\leq u_-\equiv 0$ on $\mathbb{T}$, then
\[v_i(x)=\min_{x\in \mathbb{T}}\{v_i(x),u_-(x)\}.\]
 By Lemma \ref{crogrp}, $Dv_i=Du_-$ on $\A_{v_i}=\mathbb{T}$. It gives rise to $Dv_i\equiv 0$. Moreover, $v_i(x)\equiv u_i<0$, which contradicts that $v_i$ satisfies
 \[\frac{1}{2}|Dv_i(x)|^2+f(x)v_i(x)=0, \quad \forall x\in \mathbb{T}.\]
This completes the proof.

\subsection{ Proof of Proposition \ref{exvnish}}

Similar to Lemma \ref{exyi}, we have
\begin{lemma}\label{exyinew}
Let $\{u_i\}_{i\in I}$ be a family of continuous functions on $M$. Then
\[\sup_{i\in I}T_t^+u_i(x)=T_t^+\left(\sup_{i\in I}u_i(x)\right),\quad \forall x\in M.\]
\end{lemma}

The following results hold under the assumption (H1), (H2) and  the moderate increasing assumption $0<\frac{\partial H}{\partial u}\leq \lambda$. In this case, we have the unique backward weak KAM solution, denoted by $u_-$ (equivalently, viscosity solution) of $H(x,u,Du)=0$, if $H$ is admissible. Moreover,
\[\A=\mathcal{I}_{(u_-,u_+)}:=\{x\in M\ |\ u_-(x)=u_+(x)\},\]
where $u_+$ denotes the maximal forward weak KAM solution of $H(x,u,Du)=0$.

\begin{lemma}\label{hyxuxttttt}
For each $y\in \A$, let
\[h_y(x):=\limsup_{t\ri +\infty}h^{y,u_-(y)}(x,t),\quad \forall x\in M.\]
Then $h_y\in \cS_+$.
\end{lemma}
\begin{proof}
 Let $(x(\cdot),u(\cdot)):\R\ri M\times\R$ be a static curve with $x(0)=y$. Then $u(t)=u_-(x(t))$ for all $t\in \R$.  By Lemma \ref{uueequi}, $h^{y,u_-(y)}(\cdot,\cdot)$ is uniformly bounded on $M\times (0,+\infty)$ and the family $\{h^{y,u_-(y)}(x,t)\}_{t>0}$ is equi-Lipschitz continuous with respect to $x$. Thus, $h_y(x)$ is well defined.
 Note that for a given $t>0$, the forward semigroup $T_t^+$ satisfies
 \[\|T_t^+\varphi-T_t^+\psi\|_\infty\leq e^{\lambda t}\|\varphi-\psi\|_\infty,\]
 for any $\varphi,\psi\in C(M,\R)$. Combining with Lemma \ref{exyinew}, $T_t^+$ commutes with $\limsup$. It follows that
 \[T_t^+h_y(x)=\limsup_{s\ri +\infty}T_t^+h^{y,u_-(y)}(x,s)=\limsup_{s\ri +\infty}h^{y,u_-(y)}(x,s+t)=h_y(x),\]
 which implies  $h_y\in \cS_+$.
\end{proof}

\begin{proposition}\label{prsa}
If the forward weak KAM solution is unique, then $\tilde{\cS}_s=\tilde{\A}$.
\end{proposition}
\begin{proof}
It suffices to prove $\tilde{\A}\subseteq\tilde{\cS}_s$. Let  $(x(\cdot),u(\cdot)):\R\ri M\times\R$ be a static curve, we only need to show for each $t_1,t_2\in \R$, there holds
\[	u(t_2)=\sup_{s>0}h^{x(t_1),u(t_1)}(x(t_2),s).\]
Since the forward weak KAM solution is unique, denoting it by $u_+$, we have
\[u(t)=u_-(x(t))=u_+(x(t)),\quad \forall t\in\R.\]
By Lemma \ref{hyxuxttttt}, $h_y\in \cS_+$ for each $y\in \A$. Then $h_y=u_+$ on $M$ for each $y\in \A$. It yields for any $\tau_1,\tau_2\in \R$,
\begin{align*}
u(\tau_2)&=u_+(x(\tau_2))\\
&=h_{x(\tau_1)}(x(\tau_2))\\
&=\limsup_{s\ri +\infty}h^{x(\tau_1),u_-(x(\tau_1))}(x(\tau_2),s)\\
&=\limsup_{s\ri +\infty}h^{x(\tau_1),u(\tau_1)}(x(\tau_2),s).
\end{align*}
In particular, $u(t_1)=u_+(x(t_1))=h_{x(t_1)}(x(t_1))$. It remains to show
\[\limsup_{s\ri +\infty}h^{x(t_1),u(t_1)}(x(t_2),s)=\sup_{s>0}h^{x(t_1),u(t_1)}(x(t_2),s).\]
It is clear that the left hand side is not bigger than the right hand side. Since $h_y\in \cS_+$ for each $y\in \A$, then
\[h_y(x)=T_t^+h_y(x)=\sup_{z\in M}h^{z,h_y(z)}(x,s)\geq \sup_{s>0}h^{y,h_y(y)}(x,s),\]
which implies
\begin{align*}
\limsup_{s\ri +\infty}h^{x(t_1),u(t_1)}(x(t_2),s)&=h_{x(t_1)}(x(t_2))\\
&\geq\sup_{s>0}h^{x(t_1),h_{x(t_1)}(x(t_1))}(x(t_2),s)\\
&=\sup_{s>0}h^{x(t_1),u_+(x(t_1))}(x(t_2),s)\\
&=\sup_{s>0}h^{x(t_1),u(t_1)}(x(t_2),s).
\end{align*}
This completes the proof.
\end{proof}


\vskip 1cm
In the following, we  prove Proposition \ref{exvnish}. Note that $u^-_\lambda\equiv 0$ is the classical solution, then  $u^+_\lambda=u^-_\lambda$ is the maximal forward weak KAM solution. It is clear that  $u^+_\lambda$ converges uniformly to $u\equiv 0$ as $\lambda\ri 0^+$. Obviously, $u\equiv 0$ is a classical (also viscosity) solution of
\begin{equation}\label{dwcos}
\frac{1}{2}|Dw|^2+Dw\cdot V(x)=0,\quad x\in \mathbb{T}.
\end{equation}
The remaining proof is divided into two steps.

\vskip 0.2cm

{\it In the first step, we show that
 there is only one forward weak KAM solution  except $u^+_\lambda$.}

 The existence of the forward weak KAM solution except $u^+_\lambda$ follows from Proposition \ref{exppp} and Proposition \ref{prsa}. Let $v^+_\lambda\in \cS_+$ which is different from  $u^+_\lambda$. Thus, $v^+_\lambda\leq 0$ and there exists $x_0\in \mathbb{T}$ such that $v^+_\lambda(x_0)<0$. Considering
 \[\mathcal{I}_{v^+_\lambda}:=\{x\in \mathbb{T}\ |\ v^+_\lambda(x)=0\},\]
 By \cite[Theorem 1.3]{WWY2}, $\mathcal{I}_{v^+_\lambda}$ is invariant by $\pi\Phi_t$, where $\pi:T^*M\times\R\ri M$ denotes the standard projection. Consequently, there are several possibilities for $\mathcal{I}_{v^+_\lambda}$ restricting on a fundamental domain of $\mathbb{T}$ denoted by $[x_1,y_0)$:
 \[\{x_1\},\ \{x_2\},\ \{x_1,x_2\},\ [x_1,x_2],\ [x_2,y_0)\cup \{x_1\},\ [x_1,y_0).\]

\vskip 0.2cm

 We assert that $x_2\notin \mathcal{I}_{v^+_\lambda}$. Then $\mathcal{I}_{v^+_\lambda}=\{x_1\}$. In fact, if  $x_2\in \mathcal{I}_{v^+_\lambda}$, then $v^+_\lambda(x_2)=0$.
Let $z_0$ be the minimum point of $v^+_\lambda$. We assume $Dv^+_\lambda(x)$ exists at $x=z_0$, which means $Dv^+_\lambda(z_0)=0$. Note that there exists $x_0\in \mathbb{T}$ such that $v^+_\lambda(x_0)<0$, then $v^+_\lambda(z_0)\leq v^+_\lambda(x_0)<0$. It follows that
 \[\lambda v^+_\lambda(z_0)+\frac{1}{2}|Dv^+_\lambda(z_0)|^2+Dv^+_\lambda(z_0)\cdot V(z_0)<0,\]
 which contradicts the definition of $v^+_\lambda$. Therefore, $v^+_\lambda(x)$ is not differentiable at $x=z_0$.
 Let $\gm:[0,+\infty]\ri \mathbb{T}$ be a $(v^+_\lambda, L_\lambda,0)$-calibrated curve with $\gm(0)=z_0$, where $L_\lambda$ denotes the Legendre transformation of $H_\lambda(x,u,p):=\lambda u+\frac{1}{2}|p|^2+p\cdot V(x)$. Namely, $L_\lambda(x,u,\dot{x}):=-\lambda u+\frac{1}{2}|\dot{x}-V(x)|^2$. Let $u(t):=v^+_\lambda(\gm(t))$ for all $t\geq 0$. By Proposition \ref{pr119955}, $(\gm(\cdot),u(\cdot)):[0,+\infty)\ri M\times\R$ is positively semi-static. By definition, it is also globally minimizing.  For each $\delta>0$, there holds
 \begin{equation}\label{huawei}
 \begin{split}
 u(\delta)&=e^{-\lambda \delta}u(0)+\inf_{\xi}\int_0^{\delta} e^{\lambda s}\frac{1}{2}|\dot{\xi}(s)-V(\xi(s))|^2 ds\\
 &=\inf_{t>0}\left\{e^{-\lambda t}u(0)+\inf_{\xi}\int_0^t e^{\lambda s}\frac{1}{2}|\dot{\xi}(s)-V(\xi(s))|^2 ds\right\},
 \end{split}
 \end{equation}
 where the  infimum is taken among the Lipschitz continuous curves $\xi:[0,\delta]\ri \mathbb{T}$ with $\xi(0)=\gm(0)=z_0$ and $\xi(\delta)=\gm(\delta)$. We assume $z_0\in (x_1,x_2)$ without loss of generality, since the argument is similar for $z_0\in (x_2,y_0)$. Similar to \cite[Proposition 10]{dw}, $\gm(t)\ri x_2$ and $u(t)\ri 0$ as $t\ri +\infty$. Thus, for $\delta>0$ large enough, $\gm(\delta)\in (z_0, x_2)$. Moreover, the infimum in (\ref{huawei}) can be achieved by a curve $\bar{\xi}:[0,\delta]\ri \mathbb{T}$ satisfying $\dot{\bar{\xi}}=V(\bar{\xi})$. It gives rise to
 \begin{equation}\label{humathe}
  e^{-\lambda \delta}u(0)=\inf_{t>0}\left\{e^{-\lambda t}u(0)\right\}.
  \end{equation}
 Note that $ u(0)=v^+_\lambda(z_0)<0$, then we have a contradiction. Therefore, we complete the proof of the assertion  $x_2\notin \mathcal{I}_{v^+_\lambda}$, which means $\mathcal{I}_{v^+_\lambda}
 $ has only one possibility, namely, $\mathcal{I}_{v^+_\lambda}=\{x_1\}$.

\vskip 0.2cm

 It is easy to check $\{(x_1,0,0)\}$ is a hyperbolic fixed point of $\Phi_t$. Similar to \cite[Proposition 10]{dw},  we obtain that  $v^+_\lambda$ is the unique forward weak KAM solution except $u^+_\lambda\equiv 0$.

\vskip 0.2cm

{\it In the second step, we show that $v^+_\lambda$  converges to $v$ as $\lambda\ri 0^+$, where $v$ is different from $u\equiv 0$.}

Since $u\equiv 0$ is the maximal forward weak KAM solution, then $v^+_\lambda\leq 0$. In view of $\mathcal{I}_{v^+_\lambda}=\{x_1\}$ and $u^-_\lambda\equiv 0$, we have $v^+_\lambda(x_1)=Dv^+_\lambda(x_1)=0$ and $v^+_\lambda(x)<0$ for $x\in (x_1,y_0)$. Hence,  a direct calculation shows
 \begin{equation}\label{hump99}
Dv^+_\lambda(x)=-V(x)-\sqrt{V^2(x)-2\lambda v^+_\lambda(x)}, \quad \forall x\in [x_1,x_2].
  \end{equation}
By \ref{hump99}, we have
 \begin{equation}\label{wllx}
v^+_\lambda(x_2)=-\int_{x_1}^{x_2}V(x)dx-\int_{x_1}^{x_2}\sqrt{V^2(x)-2\lambda v^+_\lambda(x)}dx.
  \end{equation}
Since $v^+_\lambda\leq 0$, then $v^+_\lambda(x_2)$ is decreasing as $\lambda\to 0^+$. By \cite[Lemma 3.1]{dw}, the family $\{v^+_\lambda\}_{\lambda\in (0,1]}$ is uniformly bounded and equi-Lipschitz. Note that $V(x)>0$ for $x\in (x_1,x_2)$. By  the Dominated Convergence Theorem,
 \begin{equation}\label{wllx44}
v^+_\lambda(x_2)\to-2\int_{x_1}^{x_2}V(x)dx<0,\quad \text{as}\ \lambda\to 0^+.
  \end{equation}
In addition, we know that $0\equiv v^+_\lambda(x_1)\to 0$ as $\lambda\to 0^+$.

To prove the convergence, we consider the projected Mather set of  (\ref{dwcos}). The Lagrangian associated to (\ref{dwcos}) is formulated as
\[L(x,v):=\frac{1}{2}|v-V(x)|^2.\]
Let $\mathcal{M}_0$ be the projected Mather set of $L$. Note that the Mather set  is recurrent. One has $\mathcal{M}_0=\{x_1,x_2\}$. The argument above shows for each subsequence $\lambda_n$, the limit of $v^+_{\lambda_n}$ takes the same value on $\mathcal{M}_0$. It follows that  $v^+_{\lambda}$ does convergence uniformly to $v$. By (\ref{wllx44}), $v(x_2)<0$, which is different from $u\equiv 0$.

This completes the proof of Proposition \ref{exvnish}.


\vskip 1cm

\noindent {\bf Acknowledgements:} We would like to thank the referee for the careful reading of the paper and invaluable
comments which are very helpful in improving this paper.
Kaizhi Wang is supported by NSFC Grant No. 11771283, 11931016.
Lin Wang is supported by NSFC Grant No. 11790273, 11631006.
Jun Yan is supported by NSFC Grant No.  11631006, 11790273.

\appendix
\section{Generalities}\label{gene}

The contact Lagrangian $L(x,u, v)$ associated to $H(x,u,p)$ is defined by
\[
L(x,u, v):=\sup_{p\in T^*_xM}\left\{\langle v,p\rangle-H(x,u,p)\right\}.
\]
Since $H$ satisfies (H1), (H2) and (H3), we have:
\begin{itemize}
\item [\textbf{(L1)}] {\it Strict convexity}:  the Hessian  $\frac{\partial^2 L}{\partial {v}^2} (x,u,v)$ is positive definite for all $(x,u,v)\in TM\times\R$;
    \item [\textbf{(L2)}] {\it Superlinearity}: for every $(x,u)\in M\times\R$, $L(x,u,v)$ is superlinear in $v$;
\item [\textbf{(L3)}] {\it Non-increasing}: there is a constant $\lambda>0$ such that for every $(x,u,v)\in TM\times\R$,
                        \begin{equation*}
                        -\lambda\leq \frac{\partial L}{\partial u}(x,u,v)\leq 0.
                        \end{equation*}
\end{itemize}

\subsection{Action functions}

\begin{theorem}\label{IVP}
	For any given $x_0\in M$, $u_0\in\mathbb{R}$, there exist two continuous functions $h_{x_0,u_0}(x,t)$ and $h^{x_0,u_0}(x,t)$ defined on $M\times(0,+\infty)$ satisfying	
\begin{equation}\label{baacf1}
h_{x_0,u_0}(x,t)=u_0+\inf_{\substack{\gamma(0)=x_0 \\  \gamma(t)=x} }\int_0^tL\big(\gamma(\tau),h_{x_0,u_0}(\gamma(\tau),\tau),\dot{\gamma}(\tau)\big)d\tau,
\end{equation}
\begin{align}\label{2-3}
h^{x_0,u_0}(x,t)=u_0-\inf_{\substack{\gamma(t)=x_0 \\  \gamma(0)=x } }\int_0^tL\big(\gamma(\tau),h^{x_0,u_0}(\gamma(\tau),t-\tau),\dot{\gamma}(\tau)\big)d\tau,
\end{align}
where the infimums are taken among the Lipschitz continuous curves $\gamma:[0,t]\rightarrow M$.
Moreover, the infimums in (\ref{baacf1}) and \eqref{2-3} can be achieved.
If $\gamma_1$ and $\gamma_2$ are curves achieving the infimums \eqref{baacf1} and \eqref{2-3} respectively, then $\gamma_1$ and $\gamma_2$ are of class $C^1$.
Let
\begin{align*}
x_1(s)&:=\gamma_1(s),\quad u_1(s):=h_{x_0,u_0}(\gamma_1(s),s),\,\,\,\qquad  p_1(s):=\frac{\partial L}{\partial v}(\gamma_1(s),u_1(s),\dot{\gamma}_1(s)),\\
x_2(s)&:=\gamma_2(s),\quad u_2(s):=h^{x_0,u_0}(\gamma_1(s),t-s),\quad   p_2(s):=\frac{\partial L}{\partial v}(\gamma_2(s),u_2(s),\dot{\gamma}_2(s)).
\end{align*}
Then $(x_1(s),u_1(s),p_1(s))$ and $(x_2(s),u_2(s),p_2(s))$ satisfy equations \eqref{c} with
\begin{align*}
x_1(0)=x_0, \quad x_1(t)=x, \quad \lim_{s\rightarrow 0^+}u_1(s)=u_0,\\
x_2(0)=x, \quad x_2(t)=x_0, \quad \lim_{s\rightarrow t^-}u_2(s)=u_0.
\end{align*}
\end{theorem}
We call $h_{x_0,u_0}(x,t)$ (resp. $h^{x_0,u_0}(x,t)$) a forward (resp. backward) implicit action function associated with $L$
and the curves achieving the infimums in (\ref{baacf1}) (resp. \eqref{2-3}) minimizers of $h_{x_0,u_0}(x,t)$ (resp. $h^{x_0,u_0}(x,t)$).
The relation between forward and backward implicit action functions is as follows:
\begin{proposition}\label{relation} Given $x_0$, $x\in M$, $u_0$, $u\in\mathbb{R}$ and $t>0$, then $h_{x_0,u_0}(x,t)=u$ if and only if $h^{x,u}(x_0,t)=u_0$.
\end{proposition}

If $u_1\leq u_2$, then $h_{x_0,u_2}(x,t)-u_2\leq h_{x_0,u_1}(x,t)-u_1$ for all $(x,t)\in M\times (0,+\infty)$, which together with the monotonicity of $h_{x_0,u_0}(x,t)$ in $u_0$ implies
\begin{proposition}\label{nonehh}
		\[
		|h_{x_0,u}(x,t)-h_{x_0,v}(x,t)|\leq |u-v|
		\]
		for all $u$, $v\in\mathbb{R}$ and all $(x,t)\in M\times (0,+\infty)$.
\end{proposition}
\medskip
%
%
%
%
%

By Proposition \ref{nonehh}, we have
\begin{proposition}
Given $(x_0,x,t)\in M\times M\times (0,+\infty)$, $u,v\in \R$. If $u\geq v$, then $h^{x_0,u}(x,t)-h^{x_0,v}(x,t)\geq u-v$.
\end{proposition}
\begin{proof}
By reversibility, one can find $u_0,v_0\in \R$ such that
\[h_{x,u_0}(x_0,t)=u,\quad h_{x,v_0}(x_0,t)=v.\]
If $u\geq v$, then $u_0\geq v_0$. It follows from Proposition \ref{nonehh} that
\[h_{x,u_0}(x_0,t)-h_{x,v_0}(x_0,t)\leq u_0-v_0,\]
which together with Proposition \ref{relation} implies
\[h^{x_0,u}(x,t)-h^{x_0,v}(x,t)\geq u-v.
\]
This completes the proof.
\end{proof}

By \cite[Proposition 3]{WWY4}, we have
\begin{proposition}\label{hupin}
For each $(x_0,u_0,x,t)\in M\times\R\times M\times (0,+\infty)$,
\begin{align}
\bar{h}_{x_0,u_0}(x,t)=-h^{x_0,-u_0}(x,t),\quad \bar{h}^{x_0,u_0}(x,t)=-h_{x_0,-u_0}(x,t),
\end{align}
where $\bar{h}_{x_0,u_0}(x,t)$ and $\bar{h}^{x_0,u_0}(x,t)$ denote the forward and backward action functions with respect to $\bar{H}(x,u,p):=H(x,-u,-p)$ respectively.
\end{proposition}
For the discounted case $L(x,u,v):=-\lambda u+l(x,v)$, ${h}_{x_0,u_0}(x,t)$
 can be reduced to
 \begin{equation}
 {h}_{x_0,u_0}(x,t)=e^{-\lambda t}u_0+\inf_{\gamma}\int_0^{t} e^{\lambda s}l(\gm(s),\dot{\gm}(s)) ds,
 \end{equation}
 where the  infimum is taken among the Lipschitz continuous curves $\gm:[0,t]\ri M$ with $\gm(0)=x_0$ and $\gm(t)=x$.
  ${h}^{x_0,u_0}(x,t)$ can be reduced to
  \begin{equation}
 {h}^{x_0,u_0}(x,t)=e^{\lambda t}u_0-\inf_{\gamma}\int_0^{t} e^{\lambda s}l(\gm(s),\dot{\gm}(s)) ds,
 \end{equation}
where the  infimum is taken among the Lipschitz continuous curves $\gm:[0,t]\ri M$ with $\gm(0)=x$ and $\gm(t)=x_0$.

In \cite{WWY2}, the authors introduced
\begin{definition}[Globally minimizing curves]
	 A curve $(x(\cdot),u(\cdot)):\mathbb{R}\to M\times\mathbb{R}$ is called globally minimizing, if it is locally Lipschitz and
		 for each $t_1$, $t_2\in\mathbb{R}$ with $t_1< t_2$, there holds
		\begin{align}
		u(t_2)=h_{x(t_1),u(t_1)}(x(t_2),t_2-t_1),
		\end{align}
where $h_{\cdot,\cdot}(\cdot,\cdot)$ denotes the forward  action function associated with  $L$ (the Legendre transformation of $H$, see (\ref{baacf1}) below).
\end{definition}
 Moreover, one can define
\begin{definition}[Static curves and orbits]\label{stcont}
	A curve $(x(\cdot),u(\cdot)):\mathbb{R}\to M\times\mathbb{R}$ is called static, if it is globally minimizing and for each $t_1, t_2\in\mathbb{R}$, there holds
	\begin{equation}\label{3-399}
	u(t_2)=\inf_{s>0}h_{x(t_1),u(t_1)}(x(t_2),s).
\end{equation}
If a curve $(x(\cdot),u(\cdot)):\R\ri M\times\R$  is static , then $(x(t),u(t),p(t))$ with $t\in\mathbb{R}$  is an orbit of $\Phi_t$, where $p(t)=\frac{\partial L}{\partial v}(x(t),u(t),\dot{x}(t))$. We call it  a  static  orbit of $\Phi_t$.
\end{definition}

\begin{definition}[Aubry set]\label{audeine99}
We call  the set of all static orbits the Aubry set of $H$, denoted by $\tilde{\mathcal{A}}$. We call $\mathcal{A}:=\pi\tilde{\mathcal{A}}$ the projected Aubry set.
\end{definition}

\section{Proof of Lemma \ref{ke}}\label{prke}

\begin{lemma}\label{exyi}
Let $\{u_i\}_{i\in I}$ be a family of continuous functions on $M$. Then
\[\inf_{i\in I}T_t^-u_i(x)=T_t^-\left(\inf_{i\in I}u_i(x)\right),\quad \forall x\in M.\]
\end{lemma}
\begin{proof}
Let us recall
\[T_t^-\left(\inf_{i\in I}u_i(x)\right)=\inf_{y\in M}h_{y,\inf_{i\in I}u_i(y)}(x,t).\]
Note that the monotonicity of $h_{y,u}(x,t)$ w.r.t. $u$, it follows that
\[\inf_{y\in M}h_{y,\inf_{i\in I}u_i(y)}(x,t)=\inf_{y\in M}\inf_{i\in I}h_{y,u_i(y)}(x,t).\]
Since $y$ is independent of $i$, then
\[\inf_{y\in M}\inf_{i\in I}h_{y,u_i(y)}(x,t)=\inf_{i\in I}\inf_{y\in M}h_{y,u_i(y)}(x,t)=\inf_{i\in I}T_t^-u_i(x).\]
This completes the proof of Lemma \ref{exyi}.
\end{proof}

For each $\tau\geq 0$,
let \[w(x,\tau):=\inf_{s>0}h_{x(-\tau),u(-\tau)}(x,s).\] Then  $w(x)=\inf_{\tau\geq 0}w(x,\tau)$.

\medskip
\noindent  {\bf Item (i)}:
We divide the proof of Item (i) into three steps.

\medskip
\noindent  {\bf Step 1}: We prove that for each $0\leq \tau\leq \tau'$, $x\in M$, $w(x,\tau')\leq w(x,\tau)$. Let $\Delta:=\tau'-\tau$, by Markov property, we have
\begin{align*}
w(x,\tau')&\leq \inf_{s>\Delta}h_{x(-\tau'),u(-\tau')}(x,s)\\
&\leq \inf_{s>\Delta}h_{x(-\tau),h_{x(-\tau'),u(-\tau')}(x(-\tau),\Delta)}(x,s-\Delta)\\
&=\inf_{s>\Delta}h_{x(-\tau),u(-\tau)}(x,s-\Delta)\\
&=\inf_{t>0}h_{x(-\tau),u(-\tau)}(x,t)\\
&=w(x,\tau).
\end{align*}

\medskip
\noindent  {\bf Step 2}:  We prove that the family $\{w(x,\tau)\}_{\tau\geq 0}$ is uniformly bounded. By Step 1,
\[w(x,\tau)\leq w(x,0)=\inf_{s>0}h_{x(0),u(0)}(x,s)\leq h_{x(0),u(0)}(x,1),\]
which implies $\{w(x,\tau)\}_{\tau\geq 0}$ is uniformly bounded from above.

On the other hand, by definition and Markov property, for each $s>0$,
\begin{align*}
u(0)&=\inf_{s>0}h_{x(-\tau),u(-\tau)}(x(0),s)\\
&\leq h_{x(-\tau),u(-\tau)}(x(0),s+1)\\
&\leq h_{x,h_{x(-\tau),u(-\tau)}(x,s)}(x(0),1),
\end{align*}
which together with Proposition \ref{relation} implies
\[h_{x(-\tau),u(-\tau)}(x,s)\geq h^{x(0),u(0)}(x,1).\]
By the arbitrariness of $s$, it yields $\{w(x,\tau)\}_{\tau\geq 0}$ is uniformly bounded from below.

\medskip
\noindent  {\bf Step 3}:  We prove that the family $\{w(x,\tau)\}_{\tau\geq 2}$ is  equi-Lipschitz continuous. For $\tau\geq 2$, by Markov property,
\[\inf_{s>1}h_{x(-\tau-1),u(-\tau-1)}(x,s)\leq \inf_{s>1}h_{x(-\tau),u(-\tau)}(x,s-1)=\inf_{s'>0}h_{x(-\tau),u(-\tau)}(x,s')=w(x,\tau),\]
which together with the definition of $w(x,\tau)$ implies for each $x\in M$,
\[w_1(x,\tau):=\inf_{s>1}h_{x(-\tau-1),u(-\tau-1)}(x,s)\leq w(x,\tau)\leq \inf_{s>1}h_{x(-\tau),u(-\tau)}(x,s)=:w_2(x,\tau).\]
It follows that for each $\tau\geq 2$,
\begin{align*}
|w(x,\tau)-w(y,\tau)|\leq \max\left\{|w_2(x,\tau)-w_1(y,\tau)|,|w_2(y,\tau)-w_1(x,\tau)|\right\}.
\end{align*}
Note that for each $t\geq 0$,
\begin{align*}
\inf_{s>1}h_{x(-t),u(-t)}(x,s)&=\inf_{s>1}T_1^-h_{x(-t),u(-t)}(x,s-1)\\
&=T_1^-\inf_{s>1}h_{x(-t),u(-t)}(x,s-1)\\
&=\inf_{z\in M}h_{z,\inf_{s>1}h_{x(-t),u(-t)}(z,s-1)}(x,1)\\
&=\inf_{z\in M}h_{z,\inf_{s'>0}h_{x(-t),u(-t)}(z,s')}(x,1)\\
&=\inf_{z\in M}h_{z,w(z,t)}(x,1).
\end{align*}
In particular, there hold
\[w_1(x,\tau)=\inf_{z\in M}h_{z,w(z,\tau+1)}(x,1),\quad w_2(y,\tau)=\inf_{z\in M}h_{z,w(z,\tau)}(y,1).\]
Then we have
\begin{align*}
&|w_2(x,\tau)-w_1(y,\tau)|\\
=&|\inf_{z\in M}h_{z,w(z,\tau)}(x,1)-\inf_{z\in M}h_{z,w(z,\tau+1)}(y,1)|\\
\leq &\sup_{z\in M}|h_{z,w(z,\tau)}(x,1)-h_{z,w(z,\tau+1)}(y,1)|.
\end{align*}
By Step 2 in the proof of Item (i), there exists $K>0$ independent of $t$ such that $\|w(x,t)\|_\infty\leq K$ for each $t\geq 0$. Since $h_{\cdot,\cdot}(\cdot,1)$ is Lipschitz on $M\times [-K,K]\times M$ with some Lipschitz constant $\kappa>0$.  It follows that
\[|w_2(x,\tau)-w_1(y,\tau)|\leq \kappa d(x,y), \quad \forall \tau\geq 2.\]
Similarly, we have
\[|w_2(y,\tau)-w_1(x,\tau)|\leq \kappa d(x,y), \quad \forall \tau\geq 2.\]
Thus, the family $\{w(x,\tau)\}_{\tau\geq 2}$ is  equi-Lipschitz continuous.

Combining with Step 1 and Step 2 in the proof of Item (i), it follows that
\[
w(x)=\inf_{\tau\geq 0}\inf_{s>0}h_{x(-\tau),u(-\tau)}(x,s)=\lim_{\tau\ri +\infty}w(x,\tau),\quad \forall x\in M,
\]
which is Lipschitz continuous on $M$.

\medskip
\noindent  {\bf Item (ii)}: Under the assumptions (H1)-(H2), this item follows from Proposition \ref{con} directly. In light of Remark \ref{lipcondi}, we prove it under $|\frac{\partial H}{\partial u}|\leq \lambda$ instead of (H3).

We prove $w_\infty\in \cS_-$. Namely, $T_t^-w_\infty=w_\infty$ for each $t\geq 0$. First of all, we show $w(x)\leq T^-_tw(x)$ for all $t\geq 0$ and $x\in M$.  By Lemma \ref{exyi}, for each $t\geq 0$,
\begin{align*}
T_t^-w(x)&=\inf_{\tau\geq 0}\inf_{s> 0}T_t^-h_{x(-\tau),u(-\tau)}(x,s)\\
&=\inf_{\tau\geq 0}\inf_{s> 0}h_{x(-\tau),u(-\tau)}(x,t+s)\\
&=\inf_{\tau\geq 0}\inf_{s'> t}h_{x(-\tau),u(-\tau)}(x,s')\\
&\geq \inf_{\tau\geq 0}\inf_{s> 0}h_{x(-\tau),u(-\tau)}(x,s)\\
&=w(x).
\end{align*}

It remains to verify that the family $\{T_t^-w(x)\}_{t\geq 0}$ is uniformly bounded and equi-Lipschitz continuous. Note that $w$ is Lipschitz continuous and $T_t^-w\geq w$ on $M$,  $\{T_t^-w(x)\}_{t\geq 0}$ is uniformly bounded from below. On the other hand, by Markov property and Lemma \ref{exyi},
\begin{align*}
T_t^-w(x)&=\inf_{\tau\geq 0}\inf_{s>0}h_{x(-\tau),u(-\tau)}(x,t+s)\\
&\leq \inf_{s>0}h_{x(-t),u(-t)}(x,t+s)\\
&\leq \inf_{s>0}h_{x(0),u(0)}(x,s)\\
&\leq h_{x(0),u(0)}(x,1),
\end{align*}
which implies $\{T_t^-w(x)\}_{t\geq 0}$ is uniformly bounded from above.

The equi-Lipschitz continuity follows from the boundedness of $\{T_t^-w(x)\}_{t\geq 0}$. In fact, for $t>1$,
 \begin{align*}
|T_t^-w(x)-T_t^-w(y)|&=|T_1^-\circ T_{t-1}^-w(x)-T_1^-\circ T_{t-1}^-w(y)|\\
&=|\inf_{z\in M}h_{z,T_{t-1}^-w(z)}(x,1)-\inf_{z\in M}h_{z,T_{t-1}^-w(z)}(y,1)|\\
&\leq \sup_{z\in M}|h_{z,T_{t-1}^-w(z)}(x,1)-h_{z,T_{t-1}^-w(z)}(y,1)|.
\end{align*}
There exists $K>0$ independent of $t$ such that $|T_t^-w(x)|\leq K$ for each $t\geq 0$ and $x\in M$. Since $h_{\cdot,\cdot}(\cdot,1)$ is Lipschitz on $M\times [-K,K]\times M$ with some Lipschitz constant $\kappa>0$.  It follows that
\[|T_t^-w(x)-T_t^-w(y)|\leq \kappa d(x,y), \quad \forall t> 1.\]

\medskip
\noindent  {\bf Item (iii)}: We prove that $w_\infty(x(-\tau))=u(-\tau)$ for each $\tau\geq 0$.  We divide the proof  into two steps.

\medskip
\noindent  {\bf Step 1}:  We  prove $w(x(\sigma))=u(\sigma)$ for each $\sigma\leq 0$. On one hand, since $(x(\cdot),u(\cdot))$ is a negatively semi-static curve, for all $\sigma\leq 0$, we have \[w(x(\sigma))\leq\inf_{-\tau\leq\sigma}\inf_{s>0}h_{x(-\tau),u(-\tau)}(x(\sigma),s)=\inf_{-\tau\leq\sigma}u(\sigma)=u(\sigma).\]On the other hand, we attempt to prove
$w(x(\sigma))\geq u(\sigma)$ for each $\sigma\leq 0$. Let
\[v_{\tau}(x):=\inf_{s>0}h_{x(-\tau),u(-\tau)}(x,s).\]
We first show $v_\tau(x(\sigma))\geq u(\sigma)$ for all $0\geq-\tau> \sigma$. Suppose not. There is $-\tau'> \sigma$ such that $v_{\tau'}(x(\sigma))< u(\sigma)$. Then there exists $s'>0$ such that $h_{x(-\tau'),u(-\tau')}(x(\sigma),s')<u(\sigma)$. Since $(x(\cdot),u(\cdot)):(-\infty,0]\ri M\times\R$ is a negatively semi-static curve, then  we have
\begin{align*}
u(-\tau')&=\inf_{s>0}h_{x(-\tau'),u(-\tau')}(x(-\tau'),s)\\
&\leq h_{x(-\tau'),u(-\tau')}(x(-\tau'),s'-\tau'-\sigma)\\
&\leq h_{x(\sigma),h_{x(-\tau'),u(-\tau')}(x(\sigma),s')}(x(-\tau'),-\tau'-\sigma)\\
&<h_{x(\sigma),u(\sigma)}(x(-\tau'),-\tau'-\sigma)\\
&=u(-\tau'),
\end{align*}
which is a contradiction. Therefore, by definition, for any $\sigma \leq 0$ we have
\begin{align*}
w(x(\sigma))&=\inf_{\tau\geq 0}\inf_{s>0}h_{x(-\tau),u(-\tau)}(x(\sigma),s)\\
&=\min\{\inf_{-\tau\leq \sigma}v_\tau(x(\sigma)),\inf_{0\geq -\tau> \sigma}v_\tau(x(\sigma))\}\\
&\geq\min\{\inf_{-\tau\leq \sigma}v_\tau(x(\sigma)),u(\sigma)\}.
\end{align*}
Since \[\inf_{-\tau\leq \sigma}v_\tau(x(\sigma))=\inf_{-\tau\leq \sigma}\inf_{s>0}h_{x(-\tau),u(-\tau)}(x(\sigma),s)=\inf_{-\tau\leq \sigma}u(\sigma)=u(\sigma),\] then we have $w(x(\sigma))\geq u(\sigma)$ for all $\sigma\leq 0$.

\medskip
\noindent  {\bf Step 2}:  We show $T^-_tw(x(\sigma))=w(x(\sigma))$ for each $t>0$. By the proof of Item (ii), we have $T^-_tw(x(\sigma))\geq w(x(\sigma))$ for each $t>0$. Note that for $t>0$, $\sigma\leq 0$,
\begin{align*}
T^-_tw(x(\sigma))&=\inf_{y\in M}h_{y,w(y)}(x(\sigma),t)\\
&\leq h_{x(\sigma-t),w(x(\sigma-t))}(x(\sigma),t)\\
&\leq h_{x(\sigma-t),u(\sigma-t)}(x(\sigma),t)\\
&=u(\sigma)\\
&=w(x(\sigma)).
\end{align*}
Hence  $T^-_tw(x(\sigma))=w(x(\sigma))=u(\sigma)$ for each $t>0$ and $\sigma\leq 0$. Moreover, we have
\[
w_\infty(x(-\tau))=\lim_{t\rightarrow+\infty}T^-_tw(x(-\tau))=\lim_{t\rightarrow+\infty}u(-\tau)=u(-\tau),\quad \forall \tau\geq 0.
\]
This completes the proof.

\section{On semi-static curves}\label{ossc}
\begin{proposition}\label{Mar-new}
	Given any $x$, $y$ and $z\in M$, $u_1$, $u_2$ and $u_3\in\mathbb{R}$, $t$, $s>0$, if
	\[
	h_{x,u_1}(y,t)=u_2,\quad h_{y,u_2}(z,s)=u_3,\quad h_{x,u_1}(z,t+s)=u_3,
	\]
	then
	\[
    \gamma(\sigma):=\left\{\begin{array}{ll}
    \gamma_1(\sigma),\ \ \qquad\sigma\in[0,t],\\
   \gamma_2(\sigma-t),\,\quad\sigma\in[t,t+s],
    \end{array}\right.
	\]
	is a minimizer of $h_{x,u_1}(z,t+s)$, where $\gamma_1:[0,t]\to M$ is a minimizer of $h_{x,u_1}(y,t)$ and $\gamma_2:[0,s]\to M$ is a minimizer of $h_{y,u_2}(z,s)$.
\end{proposition}

\begin{proof}
	By Markov property, we have
	\[
	h_{x,u_1}(\gamma(\sigma),\sigma)\leq h_{y,h_{x,u_1}(y,t)}(\gamma(\sigma),\sigma-t)= h_{y,u_2}(\gamma(\sigma),\sigma-t),\quad \forall \sigma\in[t,t+s].
	\]
	We assert that the above inequality is in fact an equality, i.e.,
	 \begin{align}\label{2-1}
	 h_{x,u_1}(\gamma(\sigma),\sigma)= h_{y,u_2}(\gamma(\sigma),\sigma-t),\quad \forall \sigma\in[t,t+s].
	 \end{align}
	 If the assertion is true, then
	 \begin{align*}
	 &u_1+\int_0^{t+s}L(\gamma(\sigma),h_{x,u_1}(\gamma(\sigma),\sigma),\dot{\gamma}(\sigma))d\sigma\\
=& u_1+\int_0^tL(\gamma_1(\sigma),h_{x,u_1}(\gamma_1(\sigma),\sigma),\dot{\gamma}_1(\sigma))d\sigma\\
&+\int_t^{t+s}L(\gamma(\sigma),h_{x,u_1}(\gamma(\sigma),\sigma),\dot{\gamma}(\sigma))d\sigma\\
	 =&h_{x,u_1}(y,t)+\int_t^{t+s}L(\gamma(\sigma),h_{y,u_2}(\gamma(\sigma),\sigma-t),\dot{\gamma}(\sigma))d\sigma\\
	 =&u_2+\int_0^sL(\gamma_2(\tau),h_{y,u_2}(\gamma_2(\tau),\tau),\dot{\gamma}_2(\tau))d\tau\\
	 =&h_{y,u_2}(z,s)\\
	 =&h_{x,u_1}(z,t+s),
	 \end{align*}
	 which shows that $\gamma$ is a minimizer of $h_{x,u_1}(z,t+s)$.
	
	 Therefore, we only need to show (\ref{2-1}) holds. Suppose not. There exists $t_0\in[t,t+s)$ such that
	 \[
	  h_{x,u_1}(\gamma(t_0),t_0)< h_{y,u_2}(\gamma(t_0),t_0-t).
	  \]
	 From Markov and Monotonicity properties, we get
	 \begin{align*}
	 u_3&=h_{x,u_1}(z,t+s)\\
&\leq h_{\gamma(t_0),h_{x,u_1}(\gamma(t_0),t_0)}(z,t+s-t_0)\\
&<h_{\gamma(t_0),h_{y,u_2}(\gamma(t_0),t_0-t)}(z,t+s-t_0)\\
&=h_{y,u_2}(z,s)=u_3,
	 \end{align*}
	 which is a contradiction. The proof is complete.
\end{proof}

\begin{proposition}\label{pr1188}
	Let $(x(\cdot),u(\cdot)):(-\infty,0]\ri M\times\mathbb{R}$ be a negatively semi-static curve. Let $\gm(t):=x(t)$ for each $t\in (-\infty,0]$. Then there exists $v_-\in \cS_-$ such that $\gm:(-\infty,0]\ri M$ is ($v_-,L,0$)-calibrated. Conversely, given $v_-\in \cS_-$, let $\gm: (-\infty,0]\ri M$ be ($v_-,L,0$)-calibrated and let $x(t):=\gm(t)$, $u(t):=v_-(\gm(t))$ for each $t\in (-\infty,0]$. Then $(x(\cdot),u(\cdot)):(-\infty,0]\ri M\times\mathbb{R}$ is a negatively semi-static curve.
\end{proposition}

\begin{proof}
Let $(x(\cdot),u(\cdot)):(-\infty,0]\ri M\times\mathbb{R}$ be a negatively semi-static curve. Let $\gm(t):=x(t)$ for each $t\in (-\infty,0]$ with $\gm(0)=x(0)$.  Let
\[
w(x):=\inf_{\tau\geq 0}\inf_{s>0}h_{\gm(-\tau),u(-\tau)}(x,s),\quad v_-(x):=\lim_{t\rightarrow +\infty}T_t^-w(x),\quad \forall x\in M
\]
By Lemma \ref{ke},  $v_-\in \mathcal{S}_-$ and $v_-(\gm(-\tau))=u(-\tau)$ for each $\tau\geq 0$. By the definition of $(x(\cdot),u(\cdot))$, we have
\[u(0)-u(-t)=\int_{-t}^0L(x(\tau),u(\tau),\dot{x}(\tau))d\tau,\]
which implies for each $t>0$,
\[v_-(\gm(0))-v_-(x(-t))=\int_{-t}^0L(\gm(\tau),v_-(\gm(\tau)),\dot{\gm}(\tau))d\tau.\]
It follows that $\gm:(-\infty,0]\ri M$ is ($v_-,L,0$)-calibrated.

Conversely, given $v_-\in \cS_-$, let $\gm: (-\infty,0]\ri M$ be ($v_-,L,0$)-calibrated and let $x(t):=\gm(t)$, $u(t):=v_-(\gm(t))$ for each $t\in (-\infty,0]$. For each $t_1<t_2\leq 0$, we need to prove $(x(\cdot),u(\cdot)):(-\infty,0]\ri M\times\R$ is negatively semi-static. Namely, 	
	\[
	u(t_2)=\inf_{s>0}h_{x(t_1),u(t_1)}(x(t_2),s) \quad u(t_2)=h_{x(t_1),u(t_1)}(x(t_2),t_2-t_1).
	\]
	On one hand, in view of \cite[Proposition 4.1]{WWY2},  $\big(x(t),u(t),p(t)\big)$ satisfies equations (\ref{c}) on $(-\infty,0)$, where $p(t)=\frac{\partial L}{\partial \dot{x}}(x(t),u(t),\dot{x}(t))$. By the minimality of $h_{x_0,u_0}(x,t)$, we have
	\begin{equation}\label{maff}
	u(t_2)\geq h_{x(t_1),u(t_1)}(x(t_2),t_2-t_1)\geq\inf_{s>0}h_{x(t_1),u(t_1)}(x(t_2),s).
	\end{equation}
	On the other hand, since $v_-\in \mathcal{S}_-$, then $T_s^-v_-(x)=v_-(x)$ for all $s>0$ and $x\in M$. Thus, $v_-(y)=\inf_{x\in M}h_{x,v_-(x)}(y,s)$ for all $s>0$, which implies that for each $s>0$ and $x$, $y\in M$, we have
	$v_-(y)\leq h_{x,v_-(x)}(y,s)$. In particular, $v_-(x(t_2))\leq h_{x(t_1),v_-(x(t_1))}(x(t_2),s)$ for each $s>0$ and $t_1$, $t_2<0$. It follows that
	\[
	v_-(x(t_2))\leq \inf_{s>0}h_{x(t_1),v_-(x(t_1))}(x(t_2),s).
	\]
	Since $u(t)=v_-(x(t))$ for all $t\leq 0$, we have
	\begin{align}\label{4-500}
	u(t_2)\leq \inf_{s>0}h_{x(t_1),u(t_1)}(x(t_2),s).
	\end{align}
	By (\ref{maff}) and (\ref{4-500}), we have for each $t_1<t_2\leq 0$,  \[u(t_2)=\inf_{s>0}h_{x(t_1),u(t_1)}(x(t_2),s),\quad u(t_2)=h_{x(t_1),u(t_1)}(x(t_2),t_2-t_1).\]
This completes the proof.
\end{proof}

By a similar argument as Proposition \ref{pr1188}, we have
\begin{proposition}\label{pr119955}
	Let $(x(\cdot),u(\cdot)):\R\ri M\times\mathbb{R}$ be a semi-static curve. Let $\gm(t):=x(t)$ for each $t\in \R$. Then there exist $v_-\in \cS_-$  (resp. $v_+\in \cS_+$) such that $u(t)=v_-(\gm(t))$ (resp. $u(t)=v_+(\gm(t)$)  for each $t\in \R$ and $\gm:\R\ri M$ is ($v_{-},L,0$)-calibrated (resp. ($v_{+},L,0$)-calibrated). Conversely, we suppose that there exist $v_-\in \cS_-$  (resp. $v_+\in \cS_+$), $\gm: \R\ri M$ such that $\gm$ is ($v_{-},L,0$)-calibrated (resp. ($v_{+},L,0$)-calibrated). Let $x(t):=\gm(t)$,  $u(t):=v_-(\gm(t))$ (resp. $u(t):=v_+(\gm(t))$) for each $t\in \R$. Then $(x(\cdot),u(\cdot)):\R\ri M\times\mathbb{R}$ is a semi-static curve.
\end{proposition}

\begin{proposition}\label{bounded}
	Every negatively (resp. positively) semi-static curve $(x(\cdot),u(\cdot))$ is bounded on $(-\infty,0]$ (resp. $[0,+\infty)$) with a bound only depending on $u(0)$. Moreover, Every semi-static curve $(x(\cdot),u(\cdot))$ is bounded on $\mathbb{R}$ with a bound only depending on $u(0)$.
\end{proposition}

\begin{proof}We only need to show every negatively  semi-static curve $(x(\cdot),u(\cdot))$ is bounded on $(-\infty,0]$  with a bound only depending on $u(0)$.
By the compactness of $M$, it suffices to prove $u(t)$ is bounded on $(-\infty,0]$.
	For $t>0$, by definition, we get $u(0)=h_{x(-t),u(-t)}(x(0),t)$. By the Markov property,
\[u(-t)=\inf_{s>0}h_{x(-t),u(-t)}(x(-t),s)\leq h_{x(-t),u(-t)}(x(-t),t+1)\leq h_{x(0),u(0)}(x(-t),1).\]
	Thus, $u(t)$ is bounded from above on $(-\infty,0]$.
	
On the other hand, by definition, we have \[u(0)\leq h_{x(-t),u(-t)}(x(0),1),\]
which implies $u(-t)\geq h^{x(0),u(0)}(x(-t),1)$.	Thus, $u(t)$ is bounded from below on $(-\infty,0]$.\end{proof}

\begin{proposition}\label{ooo}
	Given $(x_0,u_0)\in M\times \R$. Let $(x(\cdot),u(\cdot)):\R\ri M\times\mathbb{R}$ be a  semi-static curve with $x(0)=x_0$ and $u(0)=u_0$. Let \[p(t):=\frac{\partial L}{\partial \dot{x}}(x(t),u(t),\dot{x}(t))\] for each $t\in \R$. Then $(x(t),u(t),p(t))$ is bounded for  $t\in \R$ and the bound only depends on $x_0,u_0$.
\end{proposition}

\begin{proof}
By Proposition \ref{bounded}, $(x(t),u(t))$ is bounded for $t\in \R$ with a bound only depending on $x_0,u_0$. By Proposition \ref{pr119955} and \cite[Theorem 1.1]{WWY2}, we have
\[H(x(t),u(t),p(t))=0,\]
which together with the assumptions (H2) and (H3) implies $p(t)$ is bounded for  $t\in \R$.
\end{proof}

\begin{lemma}\label{uueequi}
Let $(x(\cdot),u(\cdot)):\R\ri M\times\mathbb{R}$ be a semi-static curve. Then for each $\delta>0$,
\begin{itemize}
\item \text{Uniform Boundedness}:   there exists a constant $K>0$ independent of $t$ such that for $t>\delta$ and each $x\in M$, $s\in \R$, $|h^{x(s),u(s)}(x,t)|\leq K$;
\item \text{Equi-Lipschitz Continuity}:  there exists a constant $\kappa>0$ independent of $t$ such that for $t> 2\delta$ and $s\in \R$, $x\mapsto h^{x(s),u(s)}(x,t)$ are $\kappa$-Lipschitz continuous on $M$.
\end{itemize}
\end{lemma}

\begin{proof}
 Since   $(x(\cdot),u(\cdot)):\R\ri M\times\mathbb{R}$ is a semi-static curve, for any $(x,t)\in M\times (0,+\infty)$ and $s\in \R$, by Markov property,
\begin{align*}
u(s-t-1)&=h^{x(s),u(s)}(x(s-t-1),t+1),\\
&\geq h^{x,h^{x(s),u(s)}(x,t)}(x(s-t-1),1),
\end{align*}
which implies
\[h^{x(s),u(s)}(x,t)\leq h_{x(s-t-1),u(s-t-1)}(x,1).\]
By Proposition \ref{bounded}, $(x(\cdot),u(\cdot)):\R\ri M\times\mathbb{R}$ is bounded.
So, $h^{x(s),u(s)}(\cdot,\cdot)$ is bounded from above on $M\times (0,+\infty)$ for each $s\in \R$.
On the other hand, by Markov property, for any $t\geq \delta$, we have
\begin{align*}
h^{x(s),u(s)}(x,t)&=\sup_{y\in M}h^{y,h^{x(s),u(s)}(y,t-\frac{\delta}{2})}(x,\frac{\delta}{2})\\
&\geq h^{x(s-t+\frac{\delta}{2}),h^{x(s),u(s)}(x(s-t+\frac{\delta}{2}),t-\frac{\delta}{2})}(x,\frac{\delta}{2})\\
&=h^{x(s-t+\frac{\delta}{2}),u(s-t+\frac{\delta}{2})}(x,\frac{\delta}{2}).
\end{align*}
By Proposition \ref{bounded}, $u\left(s-t+\frac{\delta}{2}\right)$ is bounded. Since $h^{\cdot,\cdot}(\cdot,\frac{\delta}{2})$ is locally Lipschitz on $M\times \R\times M$, then $h^{x(s),u(s)}(\cdot,\cdot)$ is bounded from below on $M\times [\delta,+\infty)$. Thus, there exists a constant $K>0$ independent of $t$ such that for $t>\delta$ and each $x\in M$, $s\in \R$, \[|h^{x(s),u(s)}(x,t)|\leq K.\]
Note that for any $t> 2\delta$, we have
\begin{align*}
&\left|h^{x(s),u(s)}(x,t)-h^{x(s),u(s)}(y,t)\right|\\
=&\left|\sup_{z\in M}h^{z,h^{x(s),u(s)}(z,t-\delta)}(x,\delta)-\sup_{z\in M}h^{z,h^{x(s),u(s)}(z,t-\delta)}(y,\delta)\right|\\
\leq &\sup_{z\in M}\left|h^{z,h^{x(s),u(s)}(z,t-\delta)}(x,\delta)-h^{z,h^{x(s),u(s)}(z,t-\delta)}(y,\delta)\right|.
\end{align*}
Since $h^{\cdot,\cdot}(\cdot,{\delta})$ is uniformly Lipschitz on $M\times [-K,K]\times M$ with some Lipschitz constant $\kappa$, then
\[
\left|h^{x(s),u(s)}(x,t)-h^{x(s),u(s)}(y,t)\right|\leq \kappa\ d(x,y), \quad \forall t> 2\delta.
\]
This completes the proof.
\end{proof}

%

\begin{remark}\label{coclass}
In classical Hamiltonian systems, the Ma\~{n}\'{e} potential (see \cite{CI} for instance) is defined by
\[\Phi(x_0,x):=\inf_{t>0}h^t(x_0,x)=\inf_{t>0}\inf_{\gm}A_L(\gm)=\inf_{t>0}\inf_{\gm}\int_0^tL(\gm(\tau),\dot{\gm}(\tau))d\tau,\]
where the infimums are taken among the absolutely continuous curves $\gamma:[0,t]\rightarrow M$ with $\gm(0)=x_0$ and $\gm(t)=x$. $h^t(x_0,x):=\inf_{\gm}A_L(\gm)$ is called Mather's action function.
Moreover, an  absolutely continuous curve $\gm:\R\ri M$  is called semi-static if for each $t_1\leq t_2$
\begin{equation}\label{clsemi}
A_L(\gm|_{[t_1,t_2]})=\Phi(\gm(t_1),\gm(t_2)).
\end{equation}
An  absolutely continuous curve $\gm:\R\ri M$ is called static if it is semi-static and for each $t_1\leq t_2$,
\begin{equation}\label{clsta}
\Phi(\gm(t_1),\gm(t_2))+\Phi(\gm(t_1),\gm(t_2))=0.
\end{equation}

In our considerations, if $H(x,u,p)$ is independent of $u$, then
\[h_{x_0,u_0}(x,t)=u_0+h^t(x_0,x),\quad h^{x_0,u_0}(x,t)=u_0-h^t(x,x_0).\]
By Definition \ref{semdepp}, we have
\[u(t_2)=u(t_1)+\Phi(x(t_1),x(t_2)),\]
which together with $\dot{u}(t)=L(x(t),u(t),\dot{x}(t))$ implies {\rm(\ref{clsemi})} holds. By Definition \ref{stcont}, we have
\[u(t_2)=u(t_1)+\Phi(x(t_1),x(t_2)),\quad u(t_1)=u(t_2)+\Phi(x(t_2),x(t_1)),\]
which implies {\rm(\ref{clsta})} holds.

Therefore, the Aubry-Mather theory developed in present paper is compatible with the classical case.
\end{remark}

\begin{remark}\label{comtopr}
If  $0<\frac{\partial H}{\partial u}\leq \lambda$, we have $\tilde{\mathcal{A}}=\tilde{\mathcal{N}}$.
In fact, by definition, a curve defined by {\rm(\ref{3-399})} is semi-static in Definition \ref{semdepp}. On the other hand, we verify the inverse implication is also true.
Let $u_-$ be the unique viscosity solution of {\rm (\ref{hj})}. Let $(x(\cdot),u(\cdot)):\mathbb{R}\to M\times\mathbb{R}$ is semi-static. By Lemma \ref{ke}, $u(t)=u_-(x(t))$ for all $t\in \R$. By \cite[Lemma 4.8]{WWY2}, $(x(\cdot),u(\cdot)):\mathbb{R}\to M\times\mathbb{R}$ satisfies {\rm (\ref{3-399})}.
\end{remark}

\section{Proof of Lemma \ref{kkkee}}\label{prkkkee}
	The proof  is divided into three steps.\\[2mm]
	{\bf Step 1}: We  prove that for each $t\in \mathbb{R}$,
	\[
	v(t)-v_0=u(t)-u_0.
	\]
	Note that $v_0\geq u_0$. It follows from Proposition \ref{nonehh} that for each $s>0$, we have
	\[
	h_{x_0,v_0}(x(t),s)-h_{x_0,u_0}(x(t),s)\leq v_0-u_0,
	\]
	it gives rise to
	\[
	v(t):=\inf_{s>0}h_{x_0,v_0}(x(t),s)\leq \inf_{s>0}h_{x_0,u_0}(x(t),s)+v_0-u_0=u(t)+v_0-u_0.
	\]
We need to show that for each $t\in \R$, the inequality above is an equality.
	By contradiction, we assume that there exist $t_0\in \mathbb{R}$, $\sigma_0,\delta>0$ such that
	\begin{equation}\label{mnnn44}
	h_{x_0,v_0}(x(t_0),\sigma_0)=u(t_0)+v_0-u_0-\delta.
	\end{equation}
	Denote $\bar{v}(t_0):=h_{x_0,v_0}(x(t_0),\sigma_0)$. Since $v_0\geq u_0$,
	\[
	\bar{v}(t_0)\geq h_{x_0,u_0}(x(t_0),\sigma_0)\geq \inf_{\tau>0}h_{x_0,u_0}(x(t_0),\tau)=u(t_0).
	\]
	Since $(x(\cdot),u(\cdot)):\R\ri M\times\R$ is static for each $t\in \R$, then $\inf_{\tau>0}h_{x(t_0),u(t_0)}(x_0,\tau)=u_0$. Hence, for each $\varepsilon>0$, one can find $\tau_0>0$ such that
	\begin{equation}\label{mnnnn}
	h_{x(t_0),u(t_0)}(x_0,\tau_0)\leq u_0+\varepsilon.
	\end{equation}
	Note that $\bar{v}(t_0)\geq u(t_0)$,
	it follows from Proposition \ref{nonehh} that
	\begin{equation}\label{mnnnn55}
	h_{x(t_0),\bar{v}(t_0)}(x_0,\tau_0)-h_{x(t_0),u(t_0)}(x_0,\tau_0)\leq \bar{v}(t_0)-u(t_0).
	\end{equation}
	Combining (\ref{mnnn44}), (\ref{mnnnn}) and (\ref{mnnnn55}), we have
	\begin{align*}
	h_{x(t_0),\bar{v}(t_0)}(x_0,\tau_0)&\leq h_{x(t_0),u(t_0)}(x_0,\tau_0)+\bar{v}(t_0)-u(t_0)\\
	&\leq (u_0+\varepsilon)+(u(t_0)+v_0-u_0-\delta)-u(t_0)\\
	&=v_0+\varepsilon-\delta.
	\end{align*}
	By Markov property,  we have
	\[h_{x_0,v_0}(x_0,\tau_0+\sigma_0)\leq h_{x(t_0),h_{x_0,v_0}(x(t_0),\sigma_0)}(x_0,\tau_0)=h_{x(t_0),\bar{v}(t_0)}(x_0,\tau_0)\leq v_0+\varepsilon-\delta.\]
	Since $\varepsilon>0$ is arbitrary, taking $0<\varepsilon\leq \frac{\delta}{2}$, we have
	\begin{equation}\label{conrrr}
	h_{x_0,v_0}(x_0,\tau_0+\sigma_0)\leq v_0-\frac{\delta}{2}.
	\end{equation}
	
	On the other hand,  we assert
	\[
	v_0\leq h_{x_0,v_0}(x_0,\tau_0+\sigma_0),
	\]
then   $v_0\leq v_0-\frac{\delta}{2}$, which is a contradiction.

It remains to prove the assertion.
	By contradiction, we assume $v_0> h_{x_0,v_0}(x_0,\tau_0+\sigma_0)$. It follows from Markov property that
	\begin{align*}
	h_{x_0,v_0}(x_0,2(\tau_0+\sigma_0))&\leq h_{x_0,h_{x_0,v_0}(x_0,\tau_0+\sigma_0)}(x_0,\tau_0+\sigma_0)\\
&<h_{x_0,v_0}(x_0,\tau_0+\sigma_0)<v_0.
\end{align*}
	Repeating the procedure for $n$ times, it yields that \[h_{x_0,v_0}(x_0,+\infty)<h_{x_0,v_0}(x_0,\tau_0+\sigma_0)<v_0,\] which contradicts $h_{x_0,v_0}(x_0,+\infty)=v_0$.
	\\[2mm]
	{\bf Step 2}: Let $v_s(t)=\inf_{\tau>0}h_{x(s),v(s)}(x(t),\tau)$. We  prove $v(t)\geq v_s(t)$ for each $s$, $t\in \mathbb{R}$. For each $\tau>0$, $s\in \mathbb{R}$, we have
	\[
	h_{x(s),v(s)}(x_0,\tau)-h_{x(s),u(s)}(x_0,\tau)\leq v(s)-u(s).
	\]
	By Step 1, we have  $v(s)-u(s)=v_0-u_0\geq 0$. Then
	\[h_{x(s),v(s)}(x_0,\tau)-v_0\leq  h_{x(s),u(s)}(x_0,\tau)-u_0,
	\]
which implies $h_{x(s),v(s)}(x_0,+\infty)\leq v_0$ for each $t\in \R$.
	By Proposition \ref{lem3.1}, \[h_{x(s),u(s)}(x_0,+\infty)=u_0,\] it yields that for each $\varepsilon>0$, there exists $\tau_1>0$ such that
	$h_{x(s),v(s)}(x_0,\tau_1)\leq v_0+\varepsilon$. By Markov property and Monotonicity property, we have
	\begin{align*}
	h_{x(s),v(s)}(x(t),\tau_1+\tau)&\leq h_{x_0,h_{x(s),v(s)}(x_0,\tau_1)}(x(t),\tau)\\
&\leq h_{x_0,v_0+\varepsilon}(x(t),\tau)\\
&\leq h_{x_0,v_0}(x(t),\tau)+\varepsilon.
	\end{align*}
	It follows that
	\[
	v_s(t)\leq \inf_{\tau>0}h_{x(s),v(s)}(x(t),\tau_1+\tau)\leq \inf_{\tau>0}h_{x_0,v_0}(x(t),\tau)+\varepsilon=v(t)+\varepsilon.
	\]
	Since $\varepsilon$ is arbitrary, we have
	$v(t)\geq v_s(t)$ for each $s,t\in \mathbb{R}$.

	On the other hand, we  prove that $v(t)\leq v_s(t)$ for each $s$, $t\in \mathbb{R}$. By the definition of $v(t)$, for each $\varepsilon>0$, there exists $\tau_2>0$ such that
	\[
	h_{x_0,v_0}(x(s),\tau_2)\leq v(s)+\eps.
	\]
	By Markov property and Monotonicity, we get
	\[
	h_{x_0,v_0}(x(t),\tau_2+\tau)\leq h_{x(s),h_{x_0,v_0}(x(s),\tau_2)}(x(t),\tau)\leq h_{x(s),v(s)}(x(t),\tau)+\varepsilon.
	\]
	It follows that
	\[
	v(t)\leq \inf_{\tau>0}h_{x_0,v_0}(x(t),\tau_1+\tau)\leq \inf_{\tau>0}h_{x(s),v(s)}(x(t),\tau)+\varepsilon=v_s(t)+\varepsilon.
	\]
	Since $\varepsilon$ is arbitrary, we have
	$v(t)\leq v_s(t)$ for each $s,t\in \mathbb{R}$. Therefore, we obtain that for each $s,t\in \mathbb{R} $,
	
	\begin{equation}\label{vxxx}
	v(t)=\inf_{\tau>0}h_{x(s),v(s)}(x(t),\tau).
	\end{equation}
	{\bf Step 3}: It suffices to show $(x(t),v(t))$ is a globally minimizing curve. For any $t_1$, $t_2\in\mathbb{R}$, by (\ref{vxxx}) we have
	\[
	v(t_2)=\inf_{\tau>0}h_{x(t_1),v(t_1)}(x(t_2),\tau).
	\]
	Note that $v(t)-u(t)=v_0-u_0\geq 0$ for all $t\in \mathbb{R}$. It follows from Proposition \ref{nonehh} that
\[h_{x(t_1),v(t_1)}(x(t_2),\tau)-h_{x(t_1),u(t_1)}(x(t_2),\tau)\leq v(t_1)-u(t_1),\]
which means that
\[\inf_{\tau>0}h_{x(t_1),v(t_1)}(x(t_2),\tau)-v(t_1)\leq \inf_{\tau>0}h_{x(t_1),u(t_1)}(x(t_2),\tau)-u(t_1).\]

 Since $(x(\cdot),u(\cdot)):\R\ri M\times\R$ is static, then we get
	\begin{equation}\label{v2-v1}
\begin{split}
	v(t_2)-v(t_1)&=\inf_{\tau>0}h_{x(t_1),v(t_1)}(x(t_2),\tau)-v(t_1)\\
&\leq \inf_{\tau>0}h_{x(t_1),u(t_1)}(x(t_2),\tau)-u(t_1)\\
&=u(t_2)-u(t_1).
\end{split}
	\end{equation}
	By exchanging the roles of $t_1$ and $t_2$, we have
	\begin{equation}\label{u2-u1}
	v(t_2)-v(t_1)=u(t_2)-u(t_1).
	\end{equation}
	If $t_1<t_2$, it follows from (\ref{v2-v1}) that
	\begin{align*}
	v(t_2)-v(t_1)&\leq h_{x(t_1),v(t_1)}(x(t_2),t_2-t_1)-v(t_1)\\
&\leq h_{x(t_1),u(t_1)}(x(t_2),t_2-t_1)-u(t_1)\\
&=u(t_2)-u(t_1),
	\end{align*}
	which combining with (\ref{u2-u1}) implies $v(t_2)=h_{x(t_1),v(t_1)}(x(t_2),t_2-t_1)$.

This completes the proof.

\end{document}